\documentclass[11pt,reqno]{amsart}

\usepackage{amssymb}

\usepackage[english]{babel}

\usepackage[dvipsnames,svgnames,x11names,hyperref]{xcolor}
\usepackage[hyphens]{url}
\usepackage[colorlinks=true,linkcolor=Maroon,citecolor=blue,urlcolor=blue,hypertexnames=false,linktocpage]{hyperref}
\usepackage{bookmark}
\usepackage{amsmath,thmtools,mathtools}
\mathtoolsset{showonlyrefs=true}
\usepackage{bm}
\usepackage{mathtools}

\usepackage{fancyhdr}
\usepackage{esint}
\usepackage{enumerate}

\usepackage[scr]{rsfso}
\usepackage{mathtools}
\usepackage{pictexwd,dcpic}
\usepackage{graphicx}

\usepackage[labelfont=bf, font=small]{caption}

\usepackage{tikz}
\usepackage{subcaption}

\usetikzlibrary{
	arrows.meta,
	angles,
	quotes,
	calc,
	decorations.pathreplacing
}

\newcounter{mgncount}

\declaretheorem[name=Theorem,numberwithin=section]{thm}
\declaretheorem[name=Remark,style=remark,sibling=thm]{rem}
\declaretheorem[name=Lemma,sibling=thm]{lemma}
\declaretheorem[name=Proposition,sibling=thm]{prop}

\declaretheorem[name=Definition,style=definition,sibling=thm]{defn}
\declaretheorem[name=Corollary,sibling=thm]{cor}

\numberwithin{equation}{section}

\newcommand{\ep}{\epsilon}
\newcommand{\eps}{\varepsilon}

\DeclareMathOperator{\dist}{dist}

\numberwithin{equation}{section}

\begin{document}
	
	\title[VP-MCF near spheres]{Geometric \(\bm{C^1}\)-stability of geodesic spheres under
		volume-preserving mean curvature flow}

	\author{Esther Cabezas-Rivas}
	\address{\flushleft\parbox{\linewidth}{{\bf Esther Cabezas-Rivas} \\Universitat de Val\`encia\\ Departament de Matem\`atiques\\ Av. Vicent Andrés Estellés~19\\ 46100 Burjassot\\ Spain\\ {\href{mailto:esther.cabezas-rivas@uv.es}{esther.cabezas-rivas@uv.es}}}}
	
	\author{Carlo Sinestrari}
	\address{\flushleft\parbox{\linewidth}{{\bf Carlo Sinestrari} \\Università degli Studi di Roma "Tor Vergata"\\ Dipartimento di Matematica\\Via della Ricerca Scientifica\\  00133, Roma\\ Italy\\ {\href{mailto:sinestra@mat.uniroma2.it}{sinestra@mat.uniroma2.it}}}}

	\date{\today. }

	\begin{abstract}
		We prove stability of geodesic spheres under volume-preserving mean
		curvature flow in the simply connected space forms under genuinely
			first-order geometric assumptions. The initial hypersurface is required to lie in a sufficiently thin annulus around a geodesic sphere and to have outward  normal making a small angle with the radial direction. No convexity, mean convexity, initial curvature bound, or Hölder control of the normal is assumed.

		The main ingredient is a quantitative Alexandrov reflection argument that
		converts these initial assumptions into a strong reflection property
		preserved by the flow. The resulting radial and angular control allows us
		to derive estimates for the mean curvature and the full second fundamental
		form, and hence long-time existence, eventual convexity, and smooth
		exponential convergence to a geodesic sphere. In hyperbolic space, the
		eventual convexity can be strengthened to horospherical convexity. The
		reflection threshold also yields a quantitative estimate for the
		displacement of the limiting centre.
		
		For Lipschitz initial hypersurfaces, we develop a separate positive-time
		curvature estimate that is uniform along smooth approximations. This
		produces an instantaneously smooth evolution with the same eventual
		convexity and smooth exponential convergence.

	\end{abstract}
	
	\date{\today. This work  has been partially supported by project PID2025-174652NB-I00, funded by MICIU/AEI
		/10.13039/501100011033 and by ERDF, EU. The first author has also been partially supported by project CIAICO/2023/035, funded by the Conselleria d’Educació, Cultura, Universitats i Ocupació. 
		The second author has been partially supported by MUR (Ministero dell'Universit\`a e della Ricerca) Project MatMod@TOV 2037-2027, CUP E83C23000330006 and
		by the MUR Prin 2022 Project CUP E53D23005750006. The second author is a member of INdAM (Istituto Nazionale di Alta Matematica).
	}
	
	\maketitle

	\section{Introduction and main results}

	Volume-preserving mean curvature flow is the constrained gradient flow of area under fixed enclosed volume. Its equilibria are closed constant-mean-curvature hypersurfaces, including round spheres in Euclidean space. 
	 A fundamental problem is to identify  conditions on the initial data that prevent 	 singularity formation and force the flow to exist for all time and converge to a steady state. Classical convergence results
	\cite{Hu87,EsSi98,LiVPMCF,AlFr03,CabezasRivasMiquel} typically require convexity,    pointwise or integral curvature pinching, or closeness to a sphere in a topology that quantitatively controls the variation of the normal.
	
	 In this paper, we show that the stability of geodesic spheres can be established under purely first-order geometric assumptions. The initial hypersurface is
		required only to lie in a sufficiently thin annulus around a geodesic
		sphere and to have outward unit normal making a small angle with the radial
		direction. No convexity or mean convexity is assumed, and no initial
		curvature bound or quantitative modulus of continuity of the normal is
		imposed. We prove that the corresponding volume-preserving mean curvature
		flow exists for all time, becomes strictly convex, and converges smoothly
		and exponentially to a geodesic sphere. We also extend the construction to
		Lipschitz initial hypersurfaces by deriving positive-time smoothing
		estimates that are uniform along smooth approximations.

	To state our results precisely,  let \(M\) be a smooth, closed, $n$-dimensional manifold and let $X_0:M\rightarrow\mathbb R^{n+1}$
	be a smooth embedding. We denote by \(M_0=X_0(M)\) the corresponding
	embedded hypersurface and assume that \(M_0\) bounds a domain
	\(\Omega_0\). Let \(R_0>0\) be determined by
	\(
	\operatorname{vol}(\Omega_0)
	=
	\omega_{n+1}R_0^{n+1},
	\)
	where \(\omega_{n+1}\) denotes the volume of the unit ball in
	\(\mathbb R^{n+1}\).

	We assume that \(M_0\) is \(C^1\)-close to the sphere
	\(\partial B_{R_0}(0)\) in the following sense: there exist
	\(\varepsilon_0\in(0,1)\) and \(\theta_0\in(0,\pi/2)\) such that
	\begin{equation}\label{Haus}
	R_0(1-\varepsilon_0)
	\le |x|\le
	R_0(1+\varepsilon_0)
	\qquad\text{for every }x\in M_0,
	\end{equation}
	and
	\begin{equation}\label{nor-osc}
	\left\langle
	\nu_0(x),\frac{x}{|x|}
	\right\rangle
	\ge\cos\theta_0
	\qquad\text{for every }x\in M_0,
	\end{equation}
	where \(\nu_0\) is the outward unit normal to \(M_0\).
	
	We refer to \eqref{Haus} and
	\eqref{nor-osc} as the annular-closeness and the normal-oscillation assumption, respectively. 
	They can be regarded as a geometric formulation of  \(C^1\)-closeness to \(\partial B_{R_0}(0)\).
	Indeed, \eqref{nor-osc} implies that \(M_0\) is star-shaped with respect
	to the origin, so that \(M_0\) is a radial graph over \(\mathbb S^n\);
	then \eqref{Haus} and \eqref{nor-osc} are equivalent to a bound on the height
	and gradient of its graph function.

	We consider the volume-preserving mean curvature flow
	\begin{equation}\label{vp-mcf}
	\partial_tX(p,t)
	=
	\bigl(h(t)-H(p,t)\bigr)\nu(p,t),
	\qquad p\in M,\quad t>0,
	\end{equation}
	with \(X(\cdot,0)=X_0\), where $h(t)$ is the averaged mean curvature, that is,
	\begin{equation}\label{average-H}
	h(t)
	=
	\frac{1}{|M_t|}
	\int_{M_t}H\,d\mu_t.
	\end{equation}
	Here \(M_t=X(M,t)\), while \(H\) and \(\nu\) denote its mean curvature
	and outward unit normal, respectively. Our sign convention is such that
	\(-H\nu\) points inwards on a convex hypersurface (i.e. with $\kappa_1 \geq 0$, where $\kappa_1 \leq \cdots \leq \kappa_n$ denote the principal curvatures).
	
	\medskip
	Our stability result in the Euclidean setting reads as follows.

	\begin{thm}\label{thm:main}
		For every $n\geq 1$, there exists $\delta_*=\delta_*(n)>0$ with the following property.
		Let \(M_0\subset\mathbb R^{n+1}\) satisfy \eqref{Haus} and
		\eqref{nor-osc}  with $\varepsilon_0+\theta_0<\delta_*. $ Then  \eqref{vp-mcf} satisfies the following:
		\begin{enumerate}
			\item[{\rm (a)}] It  has
			a unique smooth solution, defined for every \(t\ge0\). 
			\item[{\rm (b)}] The hypersurfaces \(M_t\)
			remain uniformly star-shaped and satisfy uniform curvature bounds. 
			\item[{\rm (c)}] There exists \(T<\infty\)  such that  \(M_t\) is strictly convex for every \(t\ge T\).
			\item[{\rm (d)}] As $t \to \infty$, 
			$
			M_t$ converges
			exponentially fast in the $C^\infty$ topology to a round sphere $ B_{R_0}(p_\infty)
			$ for some \(p_\infty\in\mathbb R^{n+1}\).
			More precisely, for every
			\(k\in\mathbb N\), there exist constants \(C_k,\lambda_k>0\) such that
			\[
			\operatorname{dist}_{C^k}
			\bigl(M_t,\partial B_{R_0}(p_\infty)\bigr)
			\le C_ke^{-\lambda_kt} \qquad  \text{	for all sufficiently large \(t\). }
			\]
			\item[{\rm (e)}] The centre of the limiting sphere satisfies
			\[
			|p_\infty|
			\le
			\frac{R_0}{2}\Phi(\varepsilon_0,\theta_0), \qquad \text{where} \quad
			\Phi(\varepsilon,\theta)
			=
			(1+\varepsilon)\sin\theta
			+
			\sqrt{
				(1+\varepsilon)^2\sin^2\theta+4\varepsilon
			}.
			\]
		\end{enumerate}
	\end{thm}
	
	In particular,
	$
	|p_\infty|
	\le
	R_0\bigl(2 \theta_0+\sqrt{\varepsilon_0}\bigr)
	$
	for sufficiently small initial parameters, giving a control on the
	displacement from the reference centre.

	Our method of proof is based on the Alexandrov reflection technique, adapted to the setting of geometric flows given by Chow and Gulliver \cite{Chow98}. The approach is therefore direct and geometric compared with some of the techniques of the previous literature, e.g. integral estimates \cite{Hu87,LiVPMCF} or the functional analytic techniques employed in \cite{EsSi98}. 
	
The Euclidean result contains the core reflection mechanism, but extending
	it to the nonflat space forms requires additional geometric input.
	Alexandrov reflection remains available after replacing affine
	hyperplanes by totally geodesic hypersurfaces. In the spherical case,
	however, the reflection construction is necessarily restricted to a
	hemisphere, while in both nonflat settings the evolution equations contain
	ambient-curvature terms that must be incorporated into the curvature
	estimates and the late-time analysis.
	
	Hence, with some additional work, we can extend our $C^1$-stability result to the non-flat setting as follows. Let \(N_K^{n+1}\) denote the simply connected space form of constant sectional curvature \(K\in\{-1, 1\}\), and fix
	\(o\in N_K^{n+1}\). Now \(B_R^K(p)\) represents the geodesic ball of radius \(R\) centred at \(p\in N_K^{n+1}\).
	Consider 
	\(M_0\subset N_K^{n+1}\)  a smooth, closed and embedded hypersurface
	enclosing a domain \(\Omega_0\), and let \(R_0>0\) be determined by
	\(
	\operatorname{vol}(\Omega_0)
	=
	\operatorname{vol}\bigl(B_{R_0}^K(o)\bigr).
	\)
	Writing
	$
	r(x)=\operatorname{dist}_{N_K}(o,x),
	$
	we assume that
	\begin{equation}\label{eq:space-form-initial-annulus}
	R_0(1-\varepsilon_0)
	\le
	r(x)
	\le
	R_0(1+\varepsilon_0)
	\qquad\text{for every }x\in M_0,
	\end{equation}
	\begin{equation}\label{eq:space-form-initial-angle}
	\left\langle\nu_0,\partial_r\right\rangle
	\ge
	\cos\theta_0
	\qquad\text{on }M_0.
	\end{equation}
	
	\begin{thm}[Nonflat space-form extension]\label{space-form-main}
		Let  $1 \leq n \in \mathbb N$, \(K\in\{-1, 1\}\), and let
		\(M_0\subset N_K^{n+1}\) satisfy
		\eqref{eq:space-form-initial-annulus} and
		\eqref{eq:space-form-initial-angle}. In the spherical case, assume also
		that
		$
		R_0(1+\varepsilon_0)<\frac{\pi}{2}.
		$ If \(\varepsilon_0\) and \(\theta_0\) are sufficiently small,  then
		conclusions \textup{(a)--(d)} of Theorem~\ref{thm:main} hold, with
		round spheres replaced by geodesic spheres in \(N_K^{n+1}\). In
		particular, the limiting hypersurface is
		$
		\partial B_{R_0}^K(p_\infty)
		$
		for some \(p_\infty\in N_K^{n+1}\). In addition, when \(K=-1\), the evolving
		hypersurfaces become strictly horospherically convex for all
		sufficiently large times. The conclusion about the centre is now
		\[
		\operatorname{dist}_{N_K}(o,p_\infty)
		\le
		\omega_K(R_0,\varepsilon_0,\theta_0),
		\]
		for a suitable positive function $\omega_K$ satisfying
		$
		\omega_K(R_0,\varepsilon_0,\theta_0)
		\to 0$
		as
		$\varepsilon_0+\theta_0\to 0$.
	\end{thm}
	
	The use of the initial assumptions in the reflection technique requires a different argument in the non-flat case, and the estimates are not as direct and explicit as in Euclidean space. This leads to the different formulation of the displacement estimate at the end of the statement.

	The geometric assumptions in the above results involve only the radial position and the radial tilt of the initial hypersurface, and therefore remain meaningful for Lipschitz radial graphs. By combining smooth approximation with uniform curvature estimates, we extend the stability result to this rough-data setting as follows.
	
	\begin{thm}[Lipschitz initial hypersurfaces]
		\label{thm:intro-lipschitz}
		Let  $1 \leq n \in \mathbb N$, \(K\in\{-1,0,1\}\), and 
		$
		M_0\subset N_K^{n+1}
		$
		be a closed embedded Lipschitz hypersurface that satisfies the geometric assumptions of Theorem~\ref{thm:main} or Theorem~\ref{space-form-main},  with the normal-oscillation condition understood almost everywhere.	Then there exists an instantaneously smooth volume-preserving mean
		curvature flow
		$
		(M_t)_{t>0},
		$
		with initial condition $M_0$ in the following continuous sense: writing $M_t$ and $M_0$ as radial graphs with respect to the reference point $o$, with radial functions $\rho(\cdot,t)$ and $\rho_0$, respectively, one has \[ \rho(\cdot,t)\longrightarrow\rho_0 \qquad\text{uniformly as }t\searrow0. \] In particular, $M_t$ converges to $M_0$ in the Hausdorff sense and encloses the same volume as $M_0$. The flow has the same eventual convexity and exponential convergence properties as in the corresponding smooth theorems.
	\end{thm}
	
	\vspace*{-0.2cm}
	
	\subsection*{Relation to previous work and main novelty}
	
    There is a rich literature on convergence for the flow \eqref{vp-mcf}.
	Huisken \cite{Hu87} proved global existence and convergence to a round
	sphere for convex hypersurfaces in $\mathbb{R}^{n+1}$, $n\geq 2$, while
	the corresponding result for curves is due to Gage \cite{Gage1986}. A
	different line of research, initiated by Escher and Simonett
	\cite{EsSi98}, replaces convexity by perturbative closeness to a sphere in
	the little Hölder space $h^{1+\alpha}$, using centre-manifold and
	maximal-regularity methods. Similar stability results were obtained under Sobolev smallness assumptions
	in \cite{AntonopoulouKaraliSigal}, under integral curvature pinching in
	\cite{LiVPMCF}, and, for strictly stable critical sets in flat tori, under
	$C^{1,1}$-small perturbations in \cite{ddkk}.  
	
	Related results in nonflat space forms rely on stronger geometric or functional hypotheses. In hyperbolic space, the first author and Miquel \cite{CabezasRivasMiquel} proved convergence from horospherically convex hypersurfaces and from $C^{1, \alpha}$-perturbations of geodesic
	spheres, while subsequent works \cite{HuangLinZhang, XuLZ, LiXu} considered
	$L^2$-almost umbilical initial hypersurfaces. Hartley \cite{Hartley} treated constrained
	curvature flows near geodesic spheres in \(\mathbb S^{n+1}\), under
	smallness assumptions in spaces  \(h^{2+\alpha}\). For curves on pinched Hadamard surfaces (i.e. allowing some variable negative ambient curvature),
		\cite[Theorem~1.3]{AlbertNiclosCabezasRivas} established stability of
		geodesic circles under $h^{1+\beta}$-small perturbations.

	All  the aforementioned results either impose bounds on the curvature, in a pointwise or integral sense, or consider perturbations of a sphere in a functional norm which implies at least $C^{1,\alpha}$ smallness.  Our hypotheses instead  control only the radial position of the
		hypersurface and the pointwise angle of its normal with the radial
		direction. In particular, they impose no direct bound on the second
		fundamental form and no Hölder control of the normal, whose spatial
		variation may be arbitrarily rapid.
		The first-order nature of our assumptions is essential. Mere $C^0$-closeness does not prevent singularities: one may attach to a sphere an arbitrarily small bubble through a sufficiently thin neck, producing a hypersurface arbitrarily close to the sphere in $C^0$ whose flow develops a neckpinch.

	Our approach is based on quantitative Alexandrov reflection, in the form developed for geometric flows by Chow and Gulliver \cite{Chow98,Chow23}. Reflection has been used mainly for expanding flows and, more recently, for constrained flows in Euclidean and hyperbolic space, in combination with convexity or related positive-curvature assumptions; see \cite{BourniEspinarMishra,Harvie,AndrewsWei21,ACW21,AndrewsWei18, wei_yang_VPF_Hyperbolic}. Here, by contrast, quantitative reflection and geometric $C^1$-closeness determine each other, up to a change of parameters. The preserved reflection threshold therefore controls both the first-order geometry and the displacement of the limiting centre from the prescribed reference centre.
	
        For more general ambient manifolds, convergence to
		constant-mean-curvature hypersurfaces has been established for suitably
		$C^2$-round or non-degenerate perturbations of small geodesic spheres
		\cite{HuiskenYau,SinestrariTenan,AlFr03,MiglioranzaThesis}. In this
		generality, reflection symmetries may be unavailable and closed
		constant-mean-curvature hypersurfaces may be unstable. The stability
		studied here is therefore tied to the geometry of the ambient space.

	\subsection*{Strategy of the proof}

 The fundamental step is to establish a quantitative relation  between our geometric assumptions and a strong Alexandrov reflection property. Proposition~\ref{reflect} shows that
	the initial annular-closeness and normal-oscillation assumptions imply reflection across hyperplanes in every direction up to a common threshold,  located at a small distance from the center of the reference sphere.
	Conversely, Corollary~\ref{cor:reflection-C1-equivalence} shows that, within the
	volume-normalized class, such a threshold controls both the radial
	oscillation and the tilt of the hypersurface.  Thus, up to a quantitative change of parameters, strong reflection and geometric $C^{1}$-closeness determine each other.
	
	The decisive advantage of the reflection formulation is its invariance under the flow. By the reflection principle of Chow and Gulliver \cite{Chow98}, the initial reflection property is preserved for all positive times. Consequently, the evolving hypersurface remains in a fixed annular region and its normal remains uniformly close to the radial direction. These preserved first-order bounds provide the geometric input needed to recover the missing curvature estimates. 
	
	More precisely, the radial and angular estimates give a two-sided bound for the support function, allowing us to adapt Tso's maximum-principle argument \cite{Tso} to control the mean curvature. This step is more delicate than in the standard convex setting, since neither $H$ nor its spatial average $h$ has a prescribed sign.
	The full second fundamental form is then controlled in
	Proposition~\ref{shape-bdd}  by applying the maximum principle to a weighted quantity as in Ecker-Huisken's interior estimates  \cite{EckerHuiskenInterior}.  Since the tilt is measured relative to the spatially varying radial direction rather than to a fixed ambient vector, its evolution produces additional radial terms. The radial and angular bounds provided by Alexandrov reflection allow us to control these contributions.
	
	Global existence is then obtained by parabolic regularity and a continuation argument.
	The area dissipation implies smooth subsequential convergence to a closed
	constant-mean-curvature hypersurface, which is a sphere by Alexandrov's
	theorem; volume preservation determines its radius.  A sufficiently late time
	slice is therefore uniformly convex, allowing the flow to enter the
	classical stability regime and giving exponential convergence to a sphere.
	Finally, the preserved reflection property passes to the limiting ball
	and yields the quantitative estimate for its centre.

	The same overall strategy applies to nonflat space forms, with the Euclidean position vector replaced by the radial conformal vector field $s_K(r)\partial_r.$  The initial reflection threshold, however, must be estimated by a different argument. We first analyse it quantitatively for an off-centred geodesic sphere and then extend the estimate to hypersurfaces satisfying the annular and normal-oscillation assumptions by a perturbative continuity argument. This yields the radial and normal estimates needed for the curvature argument. The evolution equations now contain additional ambient-curvature terms, whose control uses the same preserved first-order geometry but requires separate estimates.

Although the reflection trick is common to the three space forms, the late-time analysis is not. In hyperbolic space, smooth subsequential convergence places the flow in the strictly horospherically convex regime, and the conclusion follows from the convergence theorem of \cite{CabezasRivasMiquel}. In the spherical case, convexity alone does not provide an analogous route to convergence. We therefore develop a separate self-contained bootstrap. For $n\geq 2$, it is based on the evolution of the curvature pinching and the favourable sign of the relevant ambient-curvature contribution near a geodesic sphere; for curves, the corresponding conclusion is obtained by an energy argument.

 The passage to Lipschitz initial data requires a separate curvature argument. Smooth volume-normalized approximations retain the same uniform first-order geometric bounds, but the mean-curvature estimate used in the smooth theory is not uniform because its constant depends on the initial curvature. We therefore bypass that estimate and apply a modified weighted maximum-principle argument directly to the full second fundamental form. The resulting positive-time curvature estimates are uniform along the approximating sequence and yield smooth compactness away from the initial time. The integrability of the radial velocity near $t=0$ then shows  that the limit flow converges to the prescribed initial hypersurface and encloses the same volume. The argument of the smooth case then gives the same exponential convergence to a sphere.

\vspace*{-0.2cm}
	
	\subsection*{Organization of the paper}  
	Section~\ref{notation} collects the evolution equations and introduces the Alexandrov reflection framework. Section~\ref{C1control} establishes the quantitative correspondence between the Euclidean reflection property and the radial and normal estimates, and shows that the latter are preserved along the flow. Section~\ref{sec:curvature-control} derives the bounds for the mean curvature and the full second fundamental form. Section~\ref{LTEconv} proves long-time existence, eventual convexity, smooth exponential convergence, and the quantitative estimate for the centre of the limiting sphere in Euclidean space. Section~\ref{sec:space-forms} extends the argument to hyperbolic and spherical space forms, including the ambient-curvature terms and the separate late-time analysis required in each setting. Finally, Section~\ref{sec:lip-euclidean} treats Lipschitz initial hypersurfaces by smooth
	approximation and establishes the uniform positive-time estimates required
	for instantaneous regularization and the subsequent long-time
	behaviour.

	\section{Notation, conventions and preliminary results} \label{notation}
	
	Throughout the paper, \(M_t\) is a time-dependent closed embedded hypersurface in a space form, 
	\(g=(g_{ij})\) denotes the induced metric on
	\(M_t\), \(\nabla\) its Levi-Civita connection, and \(\Delta\) the
	Laplace--Beltrami operator. We denote by \(A=(h_i^j)\) the Weingarten
	operator with respect to the outward unit normal \(\nu\), and by
	$
	H=\operatorname{tr}_g A
	$
	the mean curvature. Our sign convention is such that the principal
	curvatures of the outward-oriented sphere \(\partial B_R(p)\) are all positive.

	Indices are raised and lowered using \(g\), and repeated indices above and below are
	summed. We write
	\[
	|A|^2=h_i^jh_j^i,
	\qquad
	\operatorname{tr}(A^3)
	=
	h_i^jh_j^kh_k^i.
	\]
	The notation \(S*T\) will denote a linear combination of contractions
	of the tensors \(S\) and \(T\) with respect to \(g\). Constants denoted
	by \(C\) may change from line to line and depend only on the fixed
	parameters and the initial hypersurface, unless otherwise stated.
	
	In the first part of the paper, we consider the flow in a Euclidean ambient space. 
	We collect the evolution equations which we will later use in the curvature
	estimates. All geometric quantities below are computed on \(M_t\), and
	the time derivative is taken in the parametrization determined by
	\eqref{vp-mcf}.
	\begin{lemma}[Evolution equations] Under the {\sc vp-mcf} \eqref{vp-mcf} in $\mathbb R^{n+1}$, the following holds:
		\begin{enumerate}
			\item[{\rm (a)}] $(\partial_t-\Delta)H=(H-h)|A|^2$. \smallskip
			\item[{\rm (b)}]
			The support function $u=\langle X,\nu\rangle$ satisfies
			$(\partial_t-\Delta)u
			=
			h-2H+|A|^2u$. \smallskip
			
			\item[{\rm (c)}] $(\partial_t-\Delta)|A|^2
			=
			-2|\nabla A|^2+2|A|^4
			-2h \operatorname{tr}(A^3).$ 
			
		\end{enumerate} \label{evol-eq}
	\end{lemma}

	 Following \cite{Chow98}, we introduce the terminology for reflection of a closed hypersurface
		across a hyperplane.
	
	\begin{defn}[Reflection across a hyperplane]\label{def:reflection}
		Let \(M\subset\mathbb R^{n+1}\) be a smooth, closed, embedded hypersurface
		bounding an open domain \(\Omega\), and let
		\[
		\Pi=\Pi_{\lambda,e}
		:=
		\left\{
		x\in\mathbb R^{n+1}:
		\langle x,e\rangle=\lambda
		\right\},
		\]
		for given \(e\in\mathbb S^n\) and \(\lambda\in\mathbb R\). We denote the two
		open half-spaces determined by \(\Pi\) by
		\[
		\Pi^+
		:=
		\left\{
		x\in\mathbb R^{n+1}:
		\langle x,e\rangle>\lambda
		\right\},
		\qquad
		\Pi^-
		:=
		\left\{
		x\in\mathbb R^{n+1}:
		\langle x,e\rangle<\lambda
		\right\}.
		\]
		\begin{enumerate}	
			\item[(a)]	The reflection across \(\Pi\) is the Euclidean isometry
			\[
			\mathcal R_{\Pi}(x)
			=
			x-2\bigl(\langle x,e\rangle-\lambda\bigr)e.
			\]
			In particular, $\mathcal R_{\Pi}(\Pi^+)=\Pi^-$, $\mathcal R_{\Pi}(\Pi^-)=\Pi^+$,
			and \(\mathcal R_{\Pi}\) fixes every point of \(\Pi\).
			
			\item[(b)] \(M\) \emph{reflects across \(\Pi\) in the direction \(e\)} if
			the reflection of the portion of \(\Omega\) lying in \(\Pi^+\) is contained
			in \(\Omega\), namely
			\begin{equation}\label{eq:reflection-domain-inclusion}
			\mathcal R_{\Pi}
			\bigl(
			\Omega\cap\Pi^+
			\bigr)
			\subset
			\Omega\cap\Pi^-.
			\end{equation}
			Equivalently, the reflected cap $\mathcal R_{\Pi}
			\bigl(
			M\cap\Pi^+
			\bigr)$ 
			does not cross \(M\) and lies in \(\overline{\Omega}\cap\Pi^-\). 
			
			\item[(c)]	 \(M\) \emph{reflects strictly across \(\Pi\)} if
			\[
			\mathcal R_{\Pi}
			\bigl(
			M\cap\Pi^+
			\bigr)
			\subset
			\Omega\cap\Pi^-
			\]
			and if the direction \(e\) is not
			tangent to \(M\) at any point of \(M\cap\Pi\).

			\item[(d)] \(M\)
			\emph{reflects strictly up to \(\Pi_{\lambda,e}\)} if it reflects strictly across every
			parallel hyperplane $\Pi_{\mu,e}$ with $\mu\ge\lambda$.
		\end{enumerate}
	\end{defn}

	In the present paper, the hyperplanes \(\Pi_{\lambda,e}\)  are always  taken with \( \lambda>0 \) and therefore
	\(\Pi^+\) is
	the side of the reflecting hyperplane which is farther from the origin.
	
	Reflection across a hyperplane is an ambient isometry and therefore
	preserves the local geometric terms in the evolution equation. Moreover, the original and reflected hypersurfaces carry the same spatially constant term $h(t)$, which cancels when their local graph equations are compared. Hence the interior and boundary tangency arguments in  \cite[Theorem 2.2]{Chow98}  apply without change, and we obtain the following result.
	\begin{thm}[Preservation of the reflection property; cf.~Chow--Gulliver]\label{thm:CG} 
		Let \(M_t\subset\mathbb R^{n+1}\) be a solution of the volume-preserving mean curvature flow \eqref{vp-mcf}, for $t \in [0,T)$. If \(M_0\) reflects strictly across (resp. up to) \(\Pi_{\lambda,e}\) for some 
		\(e\in\mathbb S^n\) and \(\lambda\in\mathbb R\), then \(M_t\) also reflects strictly across (resp. up to) \(\Pi_{\lambda,e}\), for all $t \in [0,T)$.
	\end{thm}
	
	\section{$C^1$-control along the flow via Alexandrov reflection} \label{C1control}
	
	\subsection{Alexandrov reflection property at the initial time}
	
 We first show that our assumptions on $M_0$ imply the reflection property up to a certain threshold, whose size can be estimated in terms of the constants appearing in \eqref{Haus} and \eqref{nor-osc}.
	
	\begin{prop} \label{reflect}
		Let \(c_0>0\) satisfy
		\begin{equation} \label{size_c0}
		\frac{2c_0}{R_0}
		>
		\Phi(\varepsilon_0,\theta_0)
		:=
		(1+\varepsilon_0)\sin\theta_0
		+
		\sqrt{
			(1+\varepsilon_0)^2\sin^2\theta_0
			+
			4\varepsilon_0
		}.
		\textbf{}\end{equation}
		Then \(M_0\) reflects strictly  up to $\Pi=\Pi_{c_0,e}$, for every \(e\in\mathbb S^n\). 
	\end{prop}

	\begin{proof}
		After a rotation we may assume that
		\[
		\Pi=\{x\in \mathbb R^{n+1}:\langle x,e_{n+1}\rangle=c_0\}.
		\]
		The goal is to show that the part of \(M_0\) lying above the hyperplane \(\Pi\) reflects into the
		domain enclosed by \(M_0\).
		
		It is convenient to write points $x \in \mathbb R^{n+1}$ as $x=(\bar x,x_{n+1})  \in \mathbb R^n \times \mathbb R$. So let us take any \(x=(\bar x,x_{n+1})\in M_0\) with \(x_{n+1} > c_0\), and denote by
		\[
		\widetilde x=(\bar x,2c_0-x_{n+1})
		\]
		its reflection across \(\Pi\). Suppose, in order to get a contradiction, that
		\(\widetilde x\in M_0\). By the Hausdorff closeness assumption \eqref{Haus}, we can write
		\[
		R_0^2(1-\varepsilon_0)^2
		\le |\widetilde x|^2
		=|x|^2-4c_0(x_{n+1}-c_0)
		\le R_0^2(1+\varepsilon_0)^2-4c_0(x_{n+1}-c_0).
		\]
		Hence
		\[
		4c_0(x_{n+1}-c_0)
		\le R_0^2\bigl((1+\varepsilon_0)^2-(1-\varepsilon_0)^2\bigr)
		=4\varepsilon_0 R_0^2,
		\]
		and therefore
		\[
		c_0\le x_{n+1}\le c_0+\frac{\varepsilon_0R_0^2}{c_0}.
		\]
		Consequently, $\widetilde x_{n+1}=2c_0-x_{n+1}$ satisfies
		\[
		c_1:=c_0-\frac{\varepsilon_0R_0^2}{c_0}
		\le \widetilde x_{n+1}\le c_0.
		\]
		Hence, an intersection between \(M_0\) and its reflection can only occur at points with height at least $c_1$. Observe that 	our assumption on $c_0$ implies $2 c_0 > R_0 \sqrt{4 \varepsilon_0}$, hence $c_1>0$. 
		
		We want to show that \(M_0\cap\{x_{n+1}\ge c_1\}\) is a graph over the hyperplane
		\(\{x_{n+1}=0\}\). Let \(z\in M_0\cap\{x_{n+1}\ge c_1\}\), and denote by \(\xi= z/|z|\) the radial direction. Since assumption \eqref{nor-osc} ensures that
		\[
		\left\langle \nu(z),\xi\right\rangle\ge \cos\theta_0,
		\]
		the angle between \(\nu(z)\) and \(\xi\) is at most
		\(\theta_0\). Now we can write
		\[
		\nu(z)=\cos\alpha\,\xi+\sin\alpha\,\eta,
		\]
		where \(\eta\) is a unit vector orthogonal to \(\xi\), and
		\(\alpha\in[0,\theta_0]\) denotes the angle
		between \(\nu(z)\) and \(\xi\). 
		Decomposing \(e_{n+1}\) into its components parallel and orthogonal to
		\(\xi\), we have
		\[
		e_{n+1}
		=
		\frac{z_{n+1}}{|z|}\xi
		+
		\sqrt{1-\frac{z_{n+1}^2}{|z|^2}}\,\xi^\perp,
		\]
		for some unit vector \(\xi^\perp\perp\xi\). Using Cauchy-Schwarz, we deduce
		\[
		\begin{aligned}
		\langle\nu(z),e_{n+1}\rangle
		\ge
		\frac{z_{n+1}}{|z|}\cos\alpha
		-
		\sqrt{1-\frac{z_{n+1}^2}{|z|^2}}\sin\alpha.
		\end{aligned}
		\]
		Since \(z_{n+1}>0\), the right-hand side is decreasing as a function of
		\(\alpha\in[0,\theta_0]\). Hence
		\begin{equation}\label{graph}
		\langle\nu(z),e_{n+1}\rangle
		\ge
		\frac{z_{n+1}}{|z|}\cos\theta_0
		-
		\sqrt{1-\frac{z_{n+1}^2}{|z|^2}}\sin\theta_0.
		\end{equation}
		The right-hand side is positive provided
		\[
		\frac{z_{n+1}}{|z|}>\sin\theta_0.
		\]
		Using \(z_{n+1}\ge c_1\) and \(|z|\le R_0(1+\varepsilon_0)\), and recalling the definition of $c_1$,
		we see from \eqref{graph} that $\langle\nu(z),e_{n+1}\rangle>0$ provided
		\begin{equation}\label{desired}
		c_0-\dfrac{\varepsilon_0R_0^2}{c_0}
		>
		{R_0(1+\varepsilon_0)}
		\sin\theta_0.
		\end{equation}
		Setting $y=\frac{c_0}{R_0}$, 
		we may rewrite \eqref{desired} as $y^2
		-
		(1+\varepsilon_0)\sin\theta_0\,y
		-
		\varepsilon_0
		>
		0.$ 
		The positive root of the corresponding quadratic polynomial is
		\[
		\frac{
			(1+\varepsilon_0)\sin\theta_0
			+
			\sqrt{
				(1+\varepsilon_0)^2\sin^2\theta_0
				+
				4\varepsilon_0
			}
		}{2}.
		\]
		Therefore \eqref{desired} follows from the assumed lower bound
		on \(c_0\). We have proved that
		\begin{equation}\label{graph2}
		\langle \nu(z),e_{n+1}\rangle>0
		\qquad\text{on } \qquad M_0\cap\{x_{n+1}\ge c_1\}.
		\end{equation}
		
		Now we claim that \eqref{graph2} implies
		that
		$M_0\cap\{x_{n+1}\ge c_1\}$
		is a graph over the horizontal hyperplane. Indeed, the condition $\langle\nu,e_{n+1}\rangle>0$ shows first that the vertical projection is a local diffeomorphism. To prove
		global injectivity, let \(\Omega_0\) be the bounded domain enclosed by \(M_0\)
		and suppose that a vertical line
		\[
		L_{\bar z}=\{(\bar z,s):s\in\mathbb R\}
		\]
		meets \(M_0\cap\{x_{n+1}\ge c_1\}\) in two distinct points
		\[
		z^-=(\bar z,s_-),
		\qquad
		z^+=(\bar z,s_+),
		\qquad
		c_1\le s_-<s_+.
		\]
		Both intersections are transverse. Moreover,
		\(\langle\nu,e_{n+1}\rangle>0\) means that, as the line is crossed in the
		direction of increasing \(s\), both intersections are exit points from
		\(\Omega_0\). This is impossible, since the transverse intersections of the
		boundary of a bounded domain with an oriented line alternate between entry
		points, where \(\langle\nu,e_{n+1}\rangle<0\), and exit points, where
		\(\langle\nu,e_{n+1}\rangle>0\). Any intermediate entry point would also have
		height at least \(c_1\), contradicting the positivity of
		\(\langle\nu,e_{n+1}\rangle\) on the whole portion. Thus every vertical line
		meets this portion at most once, and the portion is a graph.

		Since the vertical segment joining $x$ to $\tilde x$ is contained in
		$\{x_{n+1}\ge c_1\}$, the graph property implies that its interior
		does not meet $M_0$. Moreover, $\langle\nu(x),e_{n+1}\rangle>0$ shows
		that the segment enters $\Omega_0$ immediately below $x$. Hence its
		interior lies in $\Omega_0$, and, as $\tilde x\notin M_0$, also
		$\tilde x\in\Omega_0$. Thus
		\[
		\mathcal R_\Pi\bigl(M_0\cap\Pi^+\bigr)\subset\Omega_0\cap\Pi^-.
		\]
		
		Finally, we observe that property \eqref{graph2} implies that the normal direction to $\Pi$ is not tangent to $M_0$ at the points of $M_0 \cap \Pi$. We conclude that
		\(M_0\) reflects strictly across \(\Pi_{c_0,e}\). Since the inequality required on $c_0$ is also satisfied by the values greater than $c_0$, the property holds for these values too, and
		we have the strict reflection of $M_0$ up to \(\Pi_{c_0,e}\) as in Definition \ref{def:reflection} (d).
	\end{proof}
	
	\begin{rem}\label{rem:reflection-threshold}
		We observe that
		$
		\Phi(\varepsilon_0,\theta_0)
		\to 0$
		as 
		$\varepsilon_0+\theta_0\to 0$. Therefore we can make the reflection threshold $c_0$ as small as we wish compared to $R_0$ by imposing that $\varepsilon_0+\theta_0$ are sufficiently small. 
	\end{rem}

	\subsection{Reflection and first-order closeness to spheres} \label{subsec:reflection-C1-equivalence}
	
	Before considering the evolution, we isolate the flow-independent geometric consequences of strong Alexandrov reflection. In the previous subsection, we have shown that geometric \(C^1\)-closeness to a sphere yields the reflection property across hyperplanes with distance from the origin greater than a small threshold $c_0$. By adapting ideas from \cite{Chow98}, we now show that the converse is also true, and that the reflection up to a sufficiently small $c_0$ gives a control both on the radial oscillation and on the tilt of the hypersurface. The reflection threshold $c_0$ may therefore be regarded as a geometric measure of first-order deviation from a sphere.

	\begin{prop}[Radial control from reflection]
		\label{prop:static-radial-control}
		Let \(M=\partial\Omega\subset\mathbb R^{n+1}\) be a smooth, closed,
		connected, embedded hypersurface, with \(0\in\Omega\), and let $R_0>0$ be such that $\Omega$
		has the same volume as \(B_{R_0}(0)\). Suppose that there exists $c_0 \in (0,R_0/2)$ such that 
		\(M\) reflects across \(\Pi_{c_0,e}\) for every
		\(e\in\mathbb S^n\). Then
		\begin{equation}\label{eq:static-radial-control}
		R_0-2c_0
		\le
		|x|
		\le
		R_0+2c_0
		\qquad\text{on }M.
		\end{equation}
	\end{prop}
	
	\begin{proof}
		Set
		\[
		r_-=\min_{x\in M}|x|
		\qquad \text{and} \qquad
		r_+=\max_{x\in M}|x|,
		\]
		and choose \(x_-,x_+\in M\) so that $|x_-|=r_-$, $|x_+|=r_+$. 
		Then the reflection property of $M$ across $\Pi_{c_0,e}$, with $e$ parallel to $x_+-x_-$ (see the proof of 
		\cite[Corollary~2.9]{Chow98}) yields the inequality
			\begin{equation}\label{eq:Chow-distance-comparison}
		r_+^2-r_-^2
		\le
		2c_0\,|x_+ -x_-|.
		\end{equation}
		The triangle inequality gives
		\[
		\begin{aligned}
		|x_+-x_-|
		&\le
		|x_+|+|x_-|= 
		r_+ +r_-.
		\end{aligned}
		\]
		Therefore, \eqref{eq:Chow-distance-comparison} implies
		\[
		\begin{aligned}
		\bigl(r_+ -r_-\bigr)
		\bigl(r_++r_-\bigr)
		\le
		2c_0\bigl(r_++r_-\bigr).
		\end{aligned}
		\]
		Since \(M\) is nonempty, \(r_+>0\), and hence
		\(
		r_+ +r_->0.
		\)
		We may therefore divide by the latter and conclude that
		\begin{equation}\label{eq:radial-oscillation}
		r_+-r_-\le2c_0.
		\end{equation}
		
		It remains to locate this annulus relative to \(R_0\). Notice that
		\(
		0\in\Omega
		\)
		gives
		\[
		B_{r_-}(0)
		\subset
		\Omega
		\subset
		B_{r_+}(0).
		\]
		The volume normalization therefore implies
		$r_-\le R_0\le r_+.$ 
		Combining this and \eqref{eq:radial-oscillation}, we find
		\begin{equation}\label{eq:R0-between-radii}
		R_0-2c_0 \le r_+-2c_0\le	r_-\le R_0\le r_+ \le
		r_-+2c_0 \le
		R_0+2c_0.
		\end{equation}
		Since \(R_0>2c_0\) by assumption, the lower bound above is strictly positive. Therefore \(M\)
		cannot pass through the origin. Thus, for every \(x\in M\),
		\(
		r_-\le|x|\le r_+,
		\)
		and hence
		\[
		R_0-2c_0
		\le |x|\le
		R_0+2c_0,
		\]
		which finishes the proof.
	\end{proof}

	\begin{prop}[Normal control from reflection]
		\label{prop:static-normal-control}
		Let \(M\) and $R_0$ be as in
		Proposition~\ref{prop:static-radial-control}.  
		Suppose that $c_0 \in (0,R_0/3)$ is such that 
		\(M\) reflects strictly up to \(\Pi_{c_0,e}\) for every
		\(e\in\mathbb S^n\). Then 
		\begin{equation}\label{eq:normal-oscillation_0}
		\left\langle
		\nu,\frac{x}{|x|}
		\right\rangle^2
		\ge
		1-
		\left(
		\frac{c_0}{R_0-2c_0}
		\right)^2
		\qquad\text{on }M.
		\end{equation}
		In particular, \(M\) is star-shaped with respect to the origin.
	\end{prop}
	
	\begin{proof}
		By definition of strict reflection, for any $x \in M$ and \(e\in\mathbb S^n\) we have
		\begin{equation}\label{eq:reflection-transversality}
		e\notin T_xM
		\qquad\text{whenever}\qquad
		\langle x,e\rangle>c_0.
		\end{equation}
		
		For any $x \in M$, we decompose the position vector into its normal and
		tangential components:
		\[
		x=\langle x,\nu\rangle\nu+x^\top,
		\qquad \text{with} \quad 
		x^\top=x-\langle x,\nu\rangle\nu.
		\]
		We claim that
		\begin{equation}\label{eq:tangential-position-bound}
		|x^\top|\le c_0.
		\end{equation}
		If \(x^\top=0\), there is nothing to prove. Otherwise, set \(
		e=x^\top/|x^\top|.\) Then \(e\in T_xM\) and
		\(\langle x,e\rangle =
		|x^\top|\). If \(|x^T|>c_0\), we would have
		\[
		\langle x,e\rangle>c_0
		\qquad\text{and}\qquad
		e\in T_xM,
		\]
		contradicting \eqref{eq:reflection-transversality}. Therefore
		\eqref{eq:tangential-position-bound} holds. This is also the tangential-position estimate obtained from Alexandrov
		reflection in
		\cite[Corollary~2.4]{Chow98} and
		\cite[Corollary~19(1)]{Chow23}.

		We now use the orthogonal decomposition of \(x\). Since
		\(
		|x|^2
		=
		\langle x,\nu\rangle^2+|x^\top|^2
		\), we have
		
		\[
		\begin{aligned}
		1-
		\left\langle
		\nu,\frac{x}{|x|}
		\right\rangle^2
		=
		\frac{|x|^2-\langle x,\nu\rangle^2}{|x|^2}         
		= \frac{|x^\top|^2}{|x|^2} \le
		\left(
		\frac{c_0}{R_0-2c_0}
		\right)^2,
		\end{aligned}
		\]
		where we have applied \eqref{eq:tangential-position-bound} and Proposition~\ref{prop:static-radial-control} for the last inequality.
		
		Notice that, since \(3c_0<R_0\), then
		\(
		\frac{c_0}{R_0-2c_0}<1,
		\)
		hence the right-hand side of \eqref{eq:normal-oscillation_0} is strictly
		positive. Thus
		\(	\left\langle
		\nu,\frac{x}{|x|}
		\right\rangle
		\)
		does not vanish on \(M\). Since \(M\) is connected, this quantity has a
		constant sign. At a point where \(|x|\) attains its minimum, the outward
		unit normal points away from the origin, because \(0\in\Omega\).
		Therefore the sign is positive, and \(M\) is star-shaped.
	\end{proof}
	
	Within the class of smooth, closed, connected, embedded hypersurfaces enclosing the same volume as \(B_{R_0}(0)\) and containing the origin in the enclosed region, strong Alexandrov reflection in every direction and geometric \(C^1\)-closeness to \(\partial B_{R_0}(0)\) determine each other quantitatively, up to a change of parameters. More precisely,
	
	\begin{cor}[Reflection and geometric \(C^1\)-closeness]
		\label{cor:reflection-C1-equivalence}
		If \(M\) satisfies
		\eqref{Haus} and \eqref{nor-osc}, then every \(c_0>0\) satisfying \eqref{size_c0} 
		is a reflection threshold for \(M\). Conversely, a common reflection threshold \(c_0\), with \(3c_0<R_0\), implies the radial and normal estimates of Propositions~\ref{prop:static-radial-control} and~\ref{prop:static-normal-control}. In particular, both the relative annular width and the maximal angle between the normal and the radial direction tend to zero as \(c_0/R_0\to0\).
	\end{cor}
	
	\begin{proof}
		The first implication is Proposition~\ref{reflect}. The converse
		follows from
		Propositions~\ref{prop:static-radial-control}
		and~\ref{prop:static-normal-control}.
	\end{proof}

	The importance of the reflection formulation is that it is preserved by the evolution, by Theorem \ref{thm:CG}. Starting from the initial geometric assumptions, we obtain a strong reflection property, propagate it with the same threshold \(c_0\), and then apply the converse implication of Corollary~\ref{cor:reflection-C1-equivalence} to each time slice.

	\subsection{Preservation of the starting conditions under the flow}
	
	\begin{prop}[Annular closeness and normal oscillation for later times]
		\label{prop:later-C^1-closeness}
		Let \(M_0\subset\mathbb R^{n+1}\) satisfy \eqref{Haus} and \eqref{nor-osc} and let $M_t$ be
		the solution of the flow \eqref{vp-mcf}. Let \(c_0\) be chosen as in Proposition~\ref{reflect}, and
		assume that  \(\varepsilon_0,\theta_0\) are sufficiently small so that $R_0>3c_0$.
		
		Then, for every time for which the
		solution is defined, we have
		\begin{equation}\label{eq:hausdorff-closeness}
		R_0-2c_0
		\le |x|\le
		R_0+2c_0
		\qquad\text{on }M_t 
		\medskip
		\end{equation}
		\begin{equation}\label{eq:normal-oscillation}
		\left\langle
		\nu,\frac{x}{|x|}
		\right\rangle^2
		\ge
		1-
		\left(
		\frac{c_0}{R_0-2c_0}
		\right)^2
		\qquad\text{on }M_t.
		\end{equation}\end{prop}

	\begin{proof}
		By Proposition~\ref{reflect}, \(M_0\) reflects strictly up to
		\(\Pi_{c_0,e}\) for every \(e\in\mathbb S^n\). Then Theorem \ref{thm:CG} implies that the same threshold \(c_0\) is preserved throughout the existence interval.

		We now wish to apply the static consequences of reflection from the previous subsection. Let us check that the property $0 \in \Omega_t$ holds throughout the flow. Since
		\(0\in\Omega_0\), let
		\[
		I
		=
		\left\{
		t\ge0:
		0\in\Omega_s
		\text{ for every }0\le s\le t
		\right\}
		\]
		within the maximal existence interval. This set contains a nontrivial
		interval starting at \(t=0\). For every \(t\in I\), the hypersurface
		\(M_t\) has the same enclosed volume as \(M_0\), namely
		\[
		|\Omega_t|
		=
		|\Omega_0|
		=
		\omega_{n+1}R_0^{n+1},
		\]
		and reflects strictly up to \(\Pi_{c_0,e}\) for every direction \(e\).
		Proposition~\ref{prop:static-radial-control}
		applies to \(M_t\) with
		the same constants \(R_0\) and \(c_0\). In particular, the lower radial bound
		\[
		|x|\ge R_0-2c_0>0
		\]
		also shows that \(M_t\) cannot pass through the origin. Since
		\(0\in\Omega_0\), the origin cannot leave the enclosed region without
		first belonging to its boundary. Therefore \(0\in\Omega_t\) for every
		time in the maximal existence interval. It follows that \(I\) is the
		entire existence interval, and the above estimates
		hold throughout it.
		
		This shows that we can apply Propositions~\ref{prop:static-radial-control} and
		\ref{prop:static-normal-control} for all $t$ for which the flow exists, for the
		fixed reference radius \(R_0\) and the single reflection threshold \(c_0\). The conclusion follows.
	\end{proof}

	\begin{cor}[Control of the support function] \label{cor:support-function-control} Under the assumptions of the previous proposition, the support function $u=\langle x,\nu\rangle$ satisfies \begin{equation}\label{eq:support-function-control} R_0-3c_0 \le u \le R_0+2c_0 \end{equation} throughout the existence interval. In particular, setting $\eta_0:=3c_0,$ we have \begin{equation}\label{eq:support-function-closeness} |u-R_0|\le\eta_0 \end{equation} with $\eta_0 = \eta_0(\varepsilon_0,\theta_0) \to 0$ as $\varepsilon_0+\theta_0\to 0$.  \end{cor} 
	
	\begin{proof} First, notice that the tangential estimate
		\eqref{eq:tangential-position-bound}, applied to \(x \in M_t\), gives
		$
		|x^\top|\le c_0.
		$
		Since
		$
		x=x^\top+u\nu,
		$
		we have
		$
		u^2=|x|^2-|x^\top|^2.
		$
		Estimate~\eqref{eq:normal-oscillation} gives \(u>0\), while \eqref{eq:hausdorff-closeness} yields
		\[
		u
		\ge
		\sqrt{(R_0-2c_0)^2-c_0^2} \geq R_0 - 3 c_0,
		\]
		where the latter follows from \(3c_0<R_0\). Finally, 
		the upper bound holds because
		$
		u\le|x|\le R_0+2c_0.
		$
	\end{proof}

	\begin{rem}\label{rem:choice-reflection-threshold}
		From now on, we assume that the initial geometric parameters $\varepsilon_0,\theta_0$ are chosen sufficiently small in order that the corresponding reflection threshold satisfies
		$
		3c_0<R_0,
		$
		as required in Proposition \ref{prop:later-C^1-closeness}.
		This is certainly possible, in view of Remark \ref{rem:reflection-threshold}.
	\end{rem}

	\section{Curvature control under the flow} \label{sec:curvature-control}
	
	In this section we use the $C^1$ control from the previous section to obtain a uniform bound of the curvature.
	
	\subsection{Control of the mean curvature and the global term} \mbox{ } \medskip
	
	Let us first estimate the global term in the flow equation \eqref{vp-mcf}.
	
	\begin{lemma}[Estimate for the nonlocal term] \label{bound-h}
		Let \(\eta_0=3c_0\) be as in Corollary~\ref{cor:support-function-control}. If the initial parameters are sufficiently small that $ \eta_0<\frac{R_0}{2},$ then the averaged mean curvature defined in \eqref{average-H}
		satisfies
		\[
		\bigg|h(t) - \frac{n}{R_0}\bigg|
		\le
		\frac{\eta_0}{R_0}\max_{M_t}|H|,
		\]
		throughout the existence interval.
	\end{lemma}
	
	\begin{proof} We recall that the support function $u$  satisfies the Minkowski identity
		$$
		\int_M Hu \, d\mu = n |M|.
		$$
		Using this and the inequality $|u-R_0| \leq \eta_0$ from Corollary~\ref{cor:support-function-control}, we conclude \medskip
		\[
		\bigg|h(t) - \frac{n}{R_0}\bigg|
		= \left| \frac1{|M_t| R_0}
		\int_M H (R_0 -u) \, d\mu_t \right| \leq \frac{\eta_0}{R_0} \max_{M_t}|H|.
		\]
	\end{proof}

	\begin{rem}[The sign of the nonlocal term] The assumptions of the theorem do not determine the sign of either \(H\) or \(h\). Indeed, for a radial graph the second-order part of the mean curvature can be written in divergence form, but its contribution to the total mean curvature does not vanish: the area element introduces a nonconstant weight, and integration by parts leaves a Hessian term with no definite sign. Since our hypotheses control only the radial position and the first derivatives of the graph, this term cannot be estimated  from the initial geometric bounds. The Tso-type argument below must therefore treat the different sign configurations of \(H\) and \(h\) separately, making it more involved than in the usual sign-definite setting. 
	\end{rem}

	In the following, when considering some scalar function $f$ defined on $M_t$, we use the notation $f_{\min}:= \min_{M_t} f$ and $f_{\max} :=\max_{M_t} f$.  Differential inequalities for these quantities are understood in terms of upper and lower Dini derivatives.
	We shall use the standard scalar maximum principle without further comment.

	\begin{prop}[Mean curvature is bounded along the flow] \label{bound-H}
		If $\eps_0, \theta_0$ are small enough, then we can find a positive constant $C=C\bigl(n,R_0,\varepsilon_0,\theta_0,
		\|H(\cdot,0)\|_{L^\infty(M_0)}\bigr)$ satisfying
		\begin{equation} \label{nor-osc-t}
		|H| \leq C \qquad \text{on} \quad M_t.
		\end{equation}
	\end{prop}
	
	\begin{proof}
		Recall that, by Corollary \ref{cor:support-function-control}, after choosing
		\(\varepsilon_0,\theta_0\) sufficiently small, the support function satisfies
		$|u - R_0|\le \eta_0$,
		where \(\eta_0=\eta_0(\varepsilon_0,\theta_0)\) can be made arbitrarily small.
		
		Setting $\alpha=R_0-2\eta_0$, it holds
		\begin{equation} \label{u-cte-bdd}
		\eta_0\le u-\alpha\le 3\eta_0.
		\end{equation}
		Next we work with a Tso-type function (cf. \cite{Tso}), that is, we consider
		\[
		Z=\frac{H}{u-\alpha}.
		\]
		We now compute the evolution of \(Z\). The equations from Lemma \ref{evol-eq} give by direct computation
		\begin{equation}
		\begin{aligned}
		(\partial_t-\Delta)Z
		=
		-\frac{\alpha \, Z+h}{u-\alpha}|A|^2
		-\frac{h}{u-\alpha} Z
		+2Z^2                                      
		+\frac{2}{u-\alpha}\langle\nabla u,\nabla Z\rangle .
		\end{aligned}
		\label{equation-Z}
		\end{equation}
		We observe that $Z_{\max}$ is always positive, while $Z_{\min}$ can have any sign in general. We want to derive a two-sided bound on $Z$ using the maximum principle. The role of the terms which occur in the equation for $Z$ is different depending on the sign both of $Z$ and of $h(t)$, therefore we need to treat different cases separately.
		
		Let us first consider a time $t$ at which $h(t) \geq 0$. 	At a point where the maximum of $Z$ is attained we can write 
		\[
		\frac{d}{dt}Z_{\max}
		\le
		2Z_{\max}^2 -\frac{\alpha \, Z_{\max}}{u-\alpha}|A|^2
		\leq 2Z_{\max}^2 -\frac{\alpha}{n} Z_{\max}^3 (u - \alpha) \leq Z_{\max}^2\Big(2 - \frac{\alpha \, \eta_0}{n} Z_{\max}\Big).
		\]
		We deduce
		\begin{equation} \label{bound-Zmax}
		\frac{d}{dt}Z_{\max} \leq 0 \quad \mbox{ if $h(t)\geq 0$ \quad and }
		Z_{\max}(t)\ge  \frac{2 n}{\alpha  \eta_0}.
		\end{equation}
		
		We now estimate the derivative of $Z_{\min}$ at a time with $h(t) \geq 0$. Since we want a bound from below, we can suppose that $Z_{\min} <0$. Using \eqref{u-cte-bdd}, we have
		\begin{equation} \label{H-bounds}
		|H|_{\max}\le 3 \, \eta_0 \,  |Z|_{\max} <  3\, \eta_0 \, (Z_{\max}-Z_{\min}),
		\end{equation}
		which by Lemma \ref{bound-h} yields	
		\[
		h \leq \frac1{R_0} \Big(n + 3 \, \eta_0^2(Z_{\max} - Z_{\min})\Big).
		\]
		At a point where \(Z=Z_{\min}\), we deduce from \eqref{equation-Z}
		\[
		\frac{d}{dt}Z_{\min}
		\ge
		\frac{|A|^2}{u - \alpha}\Big(-\frac{n}{R_0} - \frac{3 \, \eta_0^2}{R_0}(Z_{\max} - Z_{\min}) - \alpha Z_{\min}\Big).
		\]
		If we choose \(\eta_0\) sufficiently small so that $\alpha R_0-3\eta_0^2> \frac {\alpha R_0}{2}$, we obtain
		\[
		\frac{d}{dt}Z_{\min}
		\ge
		\frac{|A|^2}{(u - \alpha) R_0}\Big(-n - 3 \, \eta_0^2 Z_{\max}  +\frac{\alpha R_0}{2} (- Z_{\min})\Big).
		\]
		Then if \( -Z_{\min} > \frac{2}{\alpha R_0}(n+3\eta_0^2 Z_{\max}(t))\)  the right hand side above becomes positive. We conclude that
		\begin{equation} \label{bound-Zmin}
		\frac{d}{dt}Z_{\min} \geq 0 \quad \mbox{ if $h(t)\geq 0$ \quad and }
		Z_{\min}(t)\le  -\frac{2}{\alpha R_0} \Big(n+3\eta_0^2
		Z_{\max}(t) \Big).
		\end{equation}

		Next, consider a time $t$ at which $h(t) \leq 0$. We remark that at such a time we have $H_{\min}(t)<0$ and $Z_{\min}(t)<0$. We first estimate the derivative of $Z_{\min}$. At a point where \(Z=Z_{\min}<0\), we can discard the positive terms with $-\alpha Z |A|^2$ and $Z^2$ in the evolution equation \eqref{equation-Z} and obtain
		\[\frac{d}{dt}Z_{\min}
		\ge
		-\frac{h}{u-\alpha}\big(|A|^2 + Z_{\min}\big) \ge
		-\frac{h}{u-\alpha}\bigg(\frac{H^2}{n} + Z_{\min}\bigg) \geq -\frac{h Z_{\min}}{u-\alpha} \bigg(\frac{\eta_0^2 Z_{\min}}{n} + 1\bigg).
		\]
		We deduce
		\begin{equation} \label{bound-Zmin2}
		\frac{d}{dt}Z_{\min} \geq 0 \quad \mbox{ if $h(t)\leq 0$ \quad and }
		Z_{\min}(t)\le  -\frac{n}{\eta_0^2}.
		\end{equation}

		We now consider the derivative of
		\(Z_{\max}\) under the same hypothesis that $h(t) \leq 0$.
		Suppose further that
		\(Z_{\max}(t)\geq \frac{n}{\eta_0^2} > 0\). 
		Then, at a point where \(Z=Z_{\max}\), we have
		\[
		H=(u-\alpha)Z_{\max} \geq \eta_0 Z_{\max}\geq \frac{n}{\eta_0},
		\]
		as well as
		\[
		|A|^2\ge \frac{1}{n}H^2
		\geq \frac{u - \alpha}{\eta_0}Z_{\max} \geq Z_{\max}.
		\]
		Accordingly, we can write
		\begin{align*}
		\frac{d}{dt}Z_{\max}
		&\le
		\frac{|A|^2}{u-\alpha} (- \alpha Z_{\max}-h)
		-\frac{h}{u-\alpha}Z_{\max}
		+2Z_{\max}^2 
		\\ &\leq \frac{|A|^2}{u-\alpha} (- \alpha Z_{\max}-h) -\frac{h}{u-\alpha}  |A|^2 + 2 n \frac{|A|^2}{(u - \alpha)^2}
		\\ & \leq \frac{|A|^2}{u-\alpha} \bigg(\frac{2 n}{\eta_0} -2 h  - \alpha Z_{\max}\bigg).
		\end{align*}	
		By substitution of \eqref{H-bounds} into Lemma \ref{bound-h}, we obtain
		\[
		-2h
		\le
		-\frac{2n}{R_0}
		+
		\frac{6 \eta_0^2}{R_0}
		\left(Z_{\max}- Z_{\min}\right).
		\] 
		If we choose \(\eta_0\) sufficiently small so that $\alpha R_0-6\eta_0^2> \frac {\alpha R_0}{2}$, we obtain
		\begin{align*} \label{bound-Zmax}
		\frac{d}{dt}Z_{\max}
		&\le  \frac{|A|^2}{u-\alpha}\bigg( \frac{2n}{\eta_0}
		-
		\frac{2n}{R_0} - \frac{\alpha}{2} Z_{\max}
		-
		\frac{6\eta_0^2}{R_0}Z_{\min}\bigg).
		\end{align*}
		After discarding the negative term \(-2n/R_0\) in order to simplify the constants,  we get
		\begin{equation} \label{bound-Zmax2}
		\frac{d}{dt}Z_{\max} \leq 0 \quad \mbox{ if $h(t)\leq 0$ \quad and }
		Z_{\max}(t)\ge  
		\max\left\{
		\frac{n}{\eta_0^2},
		\frac{2}{\alpha}
		\left(
		\frac{2n}{\eta_0}
		- \frac{6\eta_0^2}{R_0}Z_{\min}(t)
		\right)
		\right\}.
		\end{equation}
		
		Let us now define
		$$
		\gamma_1=\max\bigg\{ \frac{2n}{\alpha\eta_0}, \frac{n}{\eta_0^2},  \frac{4n}{\alpha \eta_0}
		\bigg\}, \qquad \delta_1= \frac{12 \eta_0^2}{\alpha R_0}.
		$$
		$$
		\gamma_2=
		\max\bigg\{ \frac{2n}{\alpha R_0}, \frac{n}{\eta_0^2}\bigg\},
		\qquad \qquad \delta_2= \frac{6 \eta_0^2}{\alpha R_0}.
		$$
		From \eqref{bound-Zmax},
		\eqref{bound-Zmin},
		\eqref{bound-Zmin2} and
		\eqref{bound-Zmax2}
		we see that, regardless of the sign of $h(t)$, we have
		\begin{equation}\label{ineq-zmax}
		\frac{d}{dt}Z_{\max} \leq 0 \quad \mbox{ if  } 
		Z_{\max} \geq \gamma_1-\delta_1 [Z_{\min}(t)]_-
		\end{equation}
		\begin{equation}\label{ineq-zmin}
		\frac{d}{dt}Z_{\min} \geq 0 \quad \mbox{ if  } 
		Z_{\min} \leq -\gamma_2-\delta_2 Z_{\max}(t),
		\end{equation}
		where $[Z_{\min}(t)]_-=\min\{Z_{\min}(t),0\}$. 
		
		Since $\eta_0$ can be made arbitrarily small by an appropriate choice of the initial parameters $\varepsilon_0, \theta_0$, we assume that it is chosen in such a way that $\delta_1<1,\delta_2<1$.  This allows us to take
		a constant  \(\Lambda>0\) such that
		$$
		-\Lambda < Z_{\min}(0) \leq Z_{\max}(0) < \Lambda,
		$$
		$$
		\Lambda> \gamma_1+\delta_1 \Lambda, \qquad
		-\Lambda<-\gamma_2-\delta_2 \Lambda.
		$$
		With this choice, we obtain   $$
		-\Lambda < Z_{\min}(t) \leq Z_{\max}(t) < \Lambda
		$$
		for every time for which the solution is defined. In fact, at a first time where either  $Z_{\min}(t)=-\Lambda$, or  $Z_{\max}(t)=\Lambda$, a standard maximum principle argument yields a contradiction with \eqref{ineq-zmax} or \eqref{ineq-zmin}.
		
		Finally, from \eqref{u-cte-bdd} we reach
		\[
		|H| = |(u-\alpha)Z| \le3\eta_0 \Lambda,
		\]
		which proves the uniform bound for the mean curvature. Therefore there exists a constant
		\[
		C=C\bigl(n,R_0,\varepsilon_0,\theta_0,
		\|H(\cdot,0)\|_{L^\infty(M_0)}\bigr)
		\]
		such that \(|H|\le C\) on \(M_t\).
	\end{proof}
	
	Notice that by the lack of convexity assumptions, we cannot conclude that $|A|^2$ is bounded from the above uniform control of the mean curvature. We need a separate extra argument inspired by \cite{EckerHuiskenInterior} to reach the conclusion.
	\begin{prop}[Shape operator is bounded along the flow] \label{shape-bdd}
		Under the same choice of parameters as in Proposition \ref{bound-H}, there exists a uniform constant bounding $|A|^2$ for all times such that the solution is defined.
	\end{prop}
	
	\begin{proof}
		For $x \in M_t$, set
		\[
		r=|x|,
		\qquad
		\omega=\frac{x}{|x|},
		\qquad
		q=\langle\nu,\omega\rangle,
		\qquad
		v=q^{-1}.
		\]
		By Proposition \ref{prop:later-C^1-closeness}, after choosing
		\(\varepsilon_0,\theta_0\) sufficiently small, there exist constants
		\(
		0<r_*<r^*<\infty\), $\sigma_*>0$,
		such that along the flow it holds
		\begin{equation}\label{eq:r-v-bounds}
		r_*\le r\le r^*,
		\qquad
		q=\langle\nu,\omega\rangle\ge\sigma_*. \qquad \text{Hence}  \qquad 	1\le v\le v_*:=\sigma_*^{-1}.
		\end{equation}

		We first record the evolution of \(v\). Since
		$u=rq$,
		the evolution from Lemma \ref{evol-eq} (b)
		and
		\[
		(\partial_t-\Delta)r
		=
		hq-\frac{n-1+q^2}{r}
		\]
		give
		\[
		\begin{aligned}
		(\partial_t-\Delta)q
		={}&
		|A|^2q
		+\frac{h}{r}(1-q^2)
		-\frac{2H}{r}
		+\frac{q(n-1+q^2)}{r^2}+
		\frac{2}{r}\langle\nabla r,\nabla q\rangle.
		\end{aligned}
		\]
		Since \(v=q^{-1}\), it follows that
		\begin{equation}\label{eq:v-schematic}
		(\partial_t-\Delta)v
		=
		-|A|^2v
		-2v^{-1}|\nabla v|^2
		+\frac{2}{r}\langle\nabla r,\nabla v\rangle + \frak R,
		\end{equation}
		where
		\[
		\frak R=
		\frac{2H}{r}v^2
		-\frac{h}{r}(v^2-1)
		-\frac{n-1}{r^2}v
		-\frac{1}{r^2v}.
		\]
		Notice  that the last two terms in the definition of \(\frak R\) are
		nonpositive. By Proposition \ref{bound-H},
		\(
		|H|+|h|\le C\).
		Together with \eqref{eq:r-v-bounds}, this implies $|\frak R|\le C$. Hereafter $C$ is a uniform constant depending on initial values, whose concrete meaning may change from line to line.

		Choose \(k>0\) so that
		\(
		kv_*^2<1,
		\)
		and define
		\begin{equation} \label{def-EH}
		\mathscr F=|A|^2\varphi(v^2), \qquad \text{with} \qquad	\varphi(s)=\frac{s}{1-ks}.
		\end{equation}
		
		By \eqref{eq:r-v-bounds}, there are positive constants
		\(c_\varphi,C_\varphi\) such that
		\begin{equation}\label{eq:phi-equivalence}
		c_\varphi
		\le\varphi(v^2)
		\le C_\varphi.
		\end{equation}
		Thus a uniform bound for \(\mathscr F\) is equivalent to a uniform bound for
		\(|A|^2\). Moreover, \(\varphi'(v^2)\) and \(\varphi''(v^2)\) are uniformly bounded. 	
		
		Combining \eqref{eq:v-schematic} and Lemma~\ref{evol-eq}(c), and carrying
		out a computation similar to \cite[Th. 3.1]{EckerHuiskenInterior} while retaining the negative
		gradient term coming from the evolution of \(|A|^2\), we obtain
		\begin{align}
		(\partial_t-\Delta)\mathscr F
		\le{}&
		-2\varphi(v^2)|\nabla A|^2
		-2k\mathscr F^2                                      -
		\frac{2k}{(1-kv^2)^2}
		|\nabla v|^2\mathscr F                               -
		\frac{2\varphi'(v^2)}{\varphi(v^2)}
		\left\langle\nabla v^2,\nabla\mathscr F\right\rangle \notag\\
		&+
		\frac{4}{r}|A|^2v\varphi'(v^2)
		\langle\nabla r,\nabla v\rangle                     -
		2h\,\operatorname{tr}(A^3)\varphi(v^2)
		+
		2|A|^2v\varphi'(v^2)\mathfrak R.
		\label{eq:evolution-mathscr-F-before-drift}
		\end{align}

		We first estimate the term containing \(\mathfrak R\). By the uniform
		boundedness of \(\mathfrak R\), \(v\), and \(\varphi'(v^2)\), and by
		\eqref{eq:phi-equivalence}, we have
		\begin{equation}\label{eq:R-contribution}
		\begin{aligned}
		2|A|^2v\varphi'(v^2)\mathfrak R
		&\le
		C|A|^2\le
		C\mathscr F.
		\end{aligned}
		\end{equation}
		
		We next rewrite the radial-gradient term as a drift term for
		\(\mathscr F\). Since
		\[
		\nabla\mathscr F
		=
		\varphi(v^2)\nabla|A|^2
		+
		2|A|^2v\varphi'(v^2)\nabla v,
		\]
		we have
		\begin{align}
		&
		\frac{4}{r}|A|^2v\varphi'(v^2)
		\langle\nabla r,\nabla v\rangle
		-
		\frac{2}{r}
		\langle\nabla r,\nabla\mathscr F\rangle                    
		=
		-\frac{2}{r}\varphi(v^2)
		\langle\nabla r,\nabla|A|^2\rangle.
		\label{eq:radial-drift-rewriting}
		\end{align}
		Using $|\nabla r|\le1$, Kato's inequality $|\nabla|A|^2|
		\le2|A||\nabla A|$ and  $
		r\ge r_*>0$,
		we obtain, for every \(\delta>0\), by Young's inequality
		\begin{align}
		\left|
		\frac{2}{r}\varphi(v^2)
		\langle\nabla r,\nabla|A|^2\rangle
		\right|
		&\le
		C\varphi(v^2)|A||\nabla A|                         \le
		\delta\varphi(v^2)|\nabla A|^2
		+
		C_\delta\mathscr F.
		\label{eq:radial-remainder-estimate}
		\end{align}
		Choose, for instance, \(\delta=1\). Combining
		\eqref{eq:radial-drift-rewriting} and
		\eqref{eq:radial-remainder-estimate} with the first term in
		\eqref{eq:evolution-mathscr-F-before-drift}, we obtain
		\begin{align*}
		&
		-2\varphi(v^2)|\nabla A|^2
		+
		\frac{4}{r}|A|^2v\varphi'(v^2)
		\langle\nabla r,\nabla v\rangle \le
		\frac{2}{r}
		\langle\nabla r,\nabla\mathscr F\rangle
		+
		C\mathscr F,
		\end{align*}
		after removing the non-positive term $-\varphi(v^2)|\nabla A|^2$. Hence
		\eqref{eq:evolution-mathscr-F-before-drift} and
		\eqref{eq:R-contribution} imply
		\begin{align}
		(\partial_t-\Delta)\mathscr F
		\le
		-2k\mathscr F^2
		+C\mathscr F                                      -
		\frac{2\varphi'(v^2)}{\varphi(v^2)}
		\left\langle\nabla v^2,\nabla\mathscr F\right\rangle
		+
		\frac{2}{r}
		\left\langle\nabla r,\nabla\mathscr F\right\rangle -
		2h \operatorname{tr}(A^3)\varphi(v^2),
		\label{eq:evolution-mathscr-F-reduced}
		\end{align}
		where we have discarded the term involving 
		\(|\nabla v|^2\mathscr F\) because it is nonpositive. 
		We still need to control the last cubic term. As $|\operatorname{tr}(A^3)|\le |A|^3$
		and \(h\) is uniformly bounded, we get
		\[
		\begin{aligned}
		\left|
		-2h\,\operatorname{tr}(A^3)\varphi(v^2)
		\right|
		&\le
		C|A|^3\varphi(v^2)=
		C\mathscr F^{3/2}\varphi(v^2)^{-1/2} \le
		C\mathscr F^{3/2} \le
		\varepsilon\mathscr F^2+C_\varepsilon,
		\end{aligned}
		\]
		which follows by Young's inequality for every \(\varepsilon>0\). Choosing \(0<\varepsilon<k\), this contribution is absorbed into the
		negative term \(-2k\mathscr F^2\) and we reach
		\begin{equation}\label{eq:final-mathscr-F-43}
		\begin{aligned}
		(\partial_t-\Delta)\mathscr F
		\le{}&
		-c\mathscr F^2+C \mathscr F+ \widetilde C                                             -
		\frac{2\varphi'(v^2)}{\varphi(v^2)}
		\left\langle\nabla v^2,\nabla\mathscr F\right\rangle
		+
		\frac{2}{r}
		\left\langle\nabla r,\nabla\mathscr F\right\rangle,
		\end{aligned}
		\end{equation}
		for some uniform constants \(c>0\) and \(C, \widetilde C<\infty\).
		
		At a spatial maximum point of \(\mathscr F\),  both gradient terms in
		\eqref{eq:final-mathscr-F-43} vanish. Therefore, 
		\begin{equation} \label{Fmax-def}
		\frac{d}{dt}\mathscr F_{\max}
		\le
		-c\mathscr F_{\max}^2
		+C \mathscr F_{\max}
		+\widetilde C.
		\end{equation}
		The right-hand side is negative whenever \(\mathscr F_{\max}\) is
		sufficiently large. The maximum principle therefore gives a uniform bound
		for \(\mathscr F\). Finally, by \eqref{eq:phi-equivalence}, we conclude
		\[
		|A|^2
		=
		\frac{\mathscr F}{\varphi(v^2)}
		\le
		\frac{\mathscr F}{c_\varphi}.
		\]
		It follows that $\sup_{M_t}|A|^2\le C$
		for every time for which the solution is defined.\end{proof}
	
	\section{Long-time existence and exponential convergence} \label{LTEconv}
	
	We complete the proof of Theorem~\ref{thm:main}.  Throughout the section, we assume that \eqref{Haus} and \eqref{nor-osc} hold, with constants satisfying the same smallness assumptions as in our previous results. We first use the curvature estimate
	to obtain long-time existence and higher-order regularity.
	
	\begin{prop}\label{prop:long-time-existence}
		The solution of \eqref{vp-mcf} exists for every
		\(t\ge0\). Moreover, for every integer \(m\ge0\) and every \(\tau>0\), there
		exists a constant \(C_{m,\tau}\) such that
		\[
		\sup_{M_t}|\nabla^mA|\le C_{m,\tau}
		\qquad\text{for every }t\ge\tau.
		\]
	\end{prop}
	
	\begin{proof}
		By Propositions~\ref{bound-H}  and ~\ref{shape-bdd}, we have uniform bounds
			\[
			|h(t)| \leq \sup_{M_t}|H|\le C, \qquad
			\sup_{M_t}|A|^2\le C
			\]
			on the maximal time interval of existence.

		The evolution equations for the covariant derivatives of the second
		fundamental form have the usual form
		\[
		(\partial_t-\Delta)\nabla^mA
		=
		\sum_{i+j+k=m}
		\nabla^iA*\nabla^jA*\nabla^kA
		+h\sum_{i+j=m}\nabla^iA*\nabla^jA,
		\]
		where the symbol \(*\) denotes a linear combination of contractions with
		respect to the induced metric. Since \(h=h(t)\) is spatially constant, no
		spatial derivatives of \(h\) occur.
		The standard Bernstein argument for mean curvature flow therefore gives, for
		every \(m\ge1\) and every \(\tau>0\),
		\[
		\sup_{M_t}|\nabla^mA|
		\le C_{m,\tau},
		\qquad t\ge\tau;
		\]
		see, for example, \cite{Hu87,EckerHuiskenInterior}. The constants
		depend on the uniform bounds for \(|A|\), \(h\), and on \(\tau\).
		
		The continuation criterion for the flow implies that a finite maximal time can
		occur only if
		\[
		\limsup_{t\nearrow T_{\max}}
		\sup_{M_t}|A|=\infty.
		\]
		This is excluded by Proposition~\ref{shape-bdd}. Hence
		$T_{\max}=\infty.$
	\end{proof}

	\begin{prop}[Eventual convexity and exponential convergence]
		\label{prop:euclidean-eventual-convexity}
		There exist \(T<\infty\) and \(p_\infty\in\mathbb R^{n+1}\) such that
		\(M_t\) is strictly convex for every \(t\ge T\), and
		\[
		M_t\longrightarrow \partial B_{R_0}(p_\infty)
		\]
		exponentially in the smooth topology as \(t\to\infty\). In addition, the limit centre point satisfies 
		\begin{equation}\label{eq:explicit-limiting-centre-bound}
		|p_\infty|
		\le
		\frac{R_0}{2}
		\Phi(\varepsilon_0,\theta_0).
		\end{equation}
	\end{prop}

	\begin{proof}
		The enclosed volume is preserved along the flow, while the area satisfies
		\begin{equation}\label{eq:area-dissipation}
		\frac{d}{dt}|M_t|
		=
		-\int_{M_t}(H-h)^2\,d\mu_t.
		\end{equation}
		Since the area is nonincreasing and bounded from below, integration of
		\eqref{eq:area-dissipation} gives
		\begin{equation}\label{eq:finite-dissipation}
		\int_0^\infty
		\int_{M_t}(H-h)^2\,d\mu_t\,dt
		<\infty.
		\end{equation}
		Consequently, there exists a sequence \(t_j\to\infty\) such that
		\begin{equation}\label{eq:vanishing-dissipation-sequence}
		\int_{M_{t_j}}(H-h)^2\,d\mu_{t_j}
		\longrightarrow0.
		\end{equation}
		
		By Proposition \ref{prop:later-C^1-closeness}, the hypersurfaces remain in a fixed compact
		annulus and are uniformly star-shaped. In particular, they can be written as
		radial graphs
		\[
		M_t
		=
		\left\{
		r(\vartheta,t)\vartheta:
		\vartheta\in\mathbb S^n
		\right\},
		\]
		where \(r\) is bounded above and below by positive constants and has a
		uniform gradient bound. Proposition~\ref{prop:long-time-existence} and the
		relations between \(A\) and the derivatives of a radial graph then give
		uniform bounds for all derivatives of \(r(\cdot,t)\), for \(t\ge1\).
		
		It follows from the Arzelà--Ascoli theorem that, after passing to a subsequence,
		\[
		r(\cdot,t_j)\longrightarrow r_\infty
		\]
		smoothly. Thus \(M_{t_j}\) converges to a smooth, closed, connected
		and embedded hypersurface \(M_\infty\).
		
		The sequence \(h(t_j)\) is bounded. After passing to a further subsequence, we
		may assume that
		\[
		h(t_j)\longrightarrow h_\infty.
		\]
		The smooth convergence and \eqref{eq:vanishing-dissipation-sequence} imply
		\[
		H_{M_\infty}\equiv h_\infty.
		\]
		Thus \(M_\infty\) is a closed embedded hypersurface with constant mean
		curvature. By Alexandrov's theorem when $n\geq2$, and by the classification of closed planar curves with constant curvature when $n=1$, it is a round sphere.
		
		The $C^2$-convergence to a round sphere implies that \(M_{t_{j}}\) is uniformly strictly convex for large enough $j$. If ${t_{j_0}}$ is a time at which our hypersurface is convex,  we conclude by Huisken's theorem \cite{Hu87}, that the entire flow \(M_t\) remains convex for $t \geq {t_{j_0}}$ and converges exponentially in the smooth topology to a round sphere as $t \to \infty$, which necessarily coincides with the limiting sphere  \(M_\infty\) along the subsequence. The corresponding convex convergence result for $n=1$ is due to Gage \cite{Gage1986}, while the exponential decay  follows from Gage's estimate for the isoperimetric deficit \cite{Gage1986}, combined with the uniform parabolic estimates for the flow; see also \cite{ChaoLingWang2013}.
		
		By the preservation of the enclosed volume, the radius of the limiting sphere is \(R_0\). Hence
		\[
		M_\infty
		=
		\partial B_{R_0}(p_\infty)
		\]
		for some \(p_\infty \in\mathbb R^{n+1}\). Finally, for every sufficiently small $\delta>0$, Proposition~\ref{reflect} applies with $c_\delta:=\frac{R_0}{2}\Phi(\varepsilon_0,\theta_0)+\delta$. Hence $c_\delta$ is a common reflection threshold for $M_0$, and the corresponding radial oscillation estimate \eqref{eq:radial-oscillation} satisfied by any $M_t$ passes to the smooth limit:  
		$$
		2 |p_\infty| = \max_{M_\infty}|x| - \min_{M_\infty}|x| \leq 2c_\delta
		$$
		Letting $\delta\searrow0$, this concludes the proof.
	\end{proof}

	\section{The result in non-flat space forms}
	\label{sec:space-forms}
	
	\subsection{Geometric \(C^1\)-closeness in space forms}
	
	We now turn to the analysis of the volume preserving mean curvature flow \eqref{vp-mcf} in space forms.
	Let \(K\in\{-1,1\}\), and let \(N_K^{n+1}\) denote the simply connected space form of constant sectional curvature \(K\):
	\[
	N_{-1}^{n+1}=\mathbb H^{n+1},
	\qquad
	N_1^{n+1}=\mathbb S^{n+1}.
	\]
	We fix a point \(o\in N_K^{n+1}\) and write
	\[
	r(x)=\dist_{N_K}(o,x).
	\]
	
	Let \(R_0>0\) be determined by
	\begin{equation}\label{eq:space-form-reference-volume}
	\operatorname{vol}(\Omega_0)
	=
	\operatorname{vol}\bigl(B_{R_0}^K(o)\bigr),
	\end{equation}
	where $\Omega_0$, as before, is the region enclosed by the initial data. 
	In the spherical case, all hypersurfaces under consideration will be
	contained in a fixed open hemisphere centred at \(o\). In particular, the
	reference radius satisfies $0<R_0<\frac{\pi}{2}$.

	We use the notation
	\begin{equation}
	\label{sine-cosine}
	s_K(r)=
	\begin{cases}
	\sinh r,&K=-1,\\
	\sin r,&K=1,
	\end{cases}
	\qquad
	c_K(r)= s_K'(r).
	\end{equation}
	These functions satisfy $c_K'=-Ks_K$, $c_K^2+Ks_K^2=1$.
	
	We denote by
	\[
	\mathbb S_o^n
	=
	\left\{
	v\in T_oN_K^{n+1}:|v|=1
	\right\}
	\]
	the unit sphere in the tangent space at \(o\).

	We assume that the initial hypersurface satisfies \eqref{eq:space-form-initial-annulus}
	and \eqref{eq:space-form-initial-angle}
	for suitably small values of \(\varepsilon_0\in(0,1)\) and
	\(\theta_0\in(0,\pi/2)\) which will be fixed in the following. We next record the interpretation of these assumptions in radial graph
	coordinates. Since
	\(
	\left\langle\nu_0,\partial_r\right\rangle>0,
	\)
	the initial hypersurface \(M_0\) is star-shaped with respect to \(o\), and hence
	can be written as
	\[
	M_0
	=
	\left\{
	\exp_o\bigl(\rho_0(\vartheta)\vartheta\bigr):
	\vartheta\in\mathbb S_o^n
	\right\}
	\]
	for a smooth function
	\(
	\rho_0:\mathbb S_o^n\rightarrow(0,\infty).
	\)
	In particular,
	\(
	r\left(
	\exp_o\bigl(\rho_0(\vartheta)\vartheta\bigr)
	\right)
	=
	\rho_0(\vartheta).
	\)
	
	Let \(\sigma\) denote the standard metric on \(\mathbb S_o^n\), and let
	\(\nabla^\sigma\) denote its Levi-Civita connection. Set
	\[
	v_0
	=
	\sqrt{
		1+
		\frac{|\nabla^\sigma\rho_0|_\sigma^2}{s_K(\rho_0)^2}
	}.
	\]
	The outward unit normal is
	\[
	\nu_0
	=
	\frac1{v_0}
	\left(
	\partial_r-
	\frac{\sigma^{ij}\partial_j \rho_0}{s_K(\rho_0)^2}\partial_i
	\right).
	\]
	It follows that
	\begin{equation}\label{eq:space-form-radial-normal}
	\left\langle\nu_0,\partial_r\right\rangle
	=
	\left(
	1+
	\frac{|\nabla^\sigma\rho_0|_\sigma^2}{s_K(\rho_0)^2}
	\right)^{-1/2}.
	\end{equation}
	Consequently,
	\begin{equation}\label{eq:space-form-angle-gradient}
	\left\langle\nu_0,\partial_r\right\rangle
	\ge\cos\theta_0
	\quad\Longleftrightarrow\quad
	|\nabla^\sigma\rho_0|_\sigma
	\le
	s_K(\rho_0)\tan\theta_0.
	\end{equation}

	\subsection{The initial reflection property}
	
	We shall use the standard linear models to represent $N_K^{n+1}$. We set
	\[
	\mathbb E_K^{n+2}=
	\begin{cases}
	\mathbb R^{n+1,1},&K=-1,\\
	\mathbb R^{n+2},&K=1,
	\end{cases}
	\]
	endowed with the Lorentzian scalar product when \(K=-1\) and the Euclidean
	scalar product when \(K=1\). We denote both scalar products by
	\(\langle\cdot,\cdot\rangle_K\). Then
	\[
	N_K^{n+1}
	=
	\left\{
	x\in\mathbb E_K^{n+2}:
	\langle x,x\rangle_K=K
	\right\},
	\]
	where, for \(K=-1\), we take the future component of the hyperboloid. In both cases, we choose our origin \(o\) to be the point whose coordinates in the ambient space are $(0,\dots,0,1)$.
	
	Given \(v\in \mathbb S_o^n \), the unit-speed geodesic issuing
	from \(o\) in the direction \(v\) is
	\begin{equation}\label{eq:space-form-geodesic}
	\gamma_v(s)=c_K(s)o+s_K(s)v.
	\end{equation}
	Its unit tangent vector is
	\begin{equation}\label{eq:space-form-geodesic-tangent}
	N_{v,s}:=\dot\gamma_v(s)=-Ks_K(s)o+c_K(s)v.
	\end{equation}
	We observe that $\langle \gamma_v(s),N_{v,s}\rangle_K=0$. We also set
	\begin{equation}\label{eq:space-form-reflecting-hypersurface}
	P_{v,s}
	=
	\left\{
	x\in N_K^{n+1}:
	\langle x,N_{v,s}\rangle_K=0
	\right\}.
	\end{equation}
	Then \(P_{v,s}\) is the totally geodesic hypersurface passing through
	\(\gamma_v(s)\) and perpendicular to \(\gamma_v\).
	We denote the two components of \(N_K^{n+1}\setminus P_{v,s}\) by
	\[
	\mathcal H_{v,s}^{+}
	=
	\left\{
	x\in N_K^{n+1}:
	\langle x,N_{v,s}\rangle_K>0
	\right\},
	\qquad
	\mathcal H_{v,s}^{-}
	=
	\left\{
	x\in N_K^{n+1}:
	\langle x,N_{v,s}\rangle_K<0
	\right\}.
	\]
	Since
	\[
	\langle o,o\rangle_K=K,\qquad
	\langle v,v\rangle_K=1,\qquad
	\langle o,v\rangle_K=0,
	\]
	equation \eqref{eq:space-form-geodesic-tangent} gives
	\begin{equation} \label{product-N}
	\langle N_{v,s},N_{v,s}\rangle_K=1,
	\qquad
	\langle N_{v,s},o\rangle_K=-s_K(s).
	\end{equation}
	In the following we will only consider $s>0$. Then we have $\langle o,N_{v,s}\rangle_K<0$,
	so \(o\in\mathcal H_{v,s}^{-}\).

	Let us denote
	by \(\mathcal R_{v,s}\) the reflection of the ambient linear space across \(P_{v,s}\), that is
	\begin{equation}\label{eq:linear-space-form-reflection}
	\mathcal R_{v,s}(x)
	=
	x-2\langle x,N_{v,s}\rangle_KN_{v,s}.
	\end{equation}
	Since \(N_{v,s}\) is a unit spacelike vector, this map preserves
	\(\langle\cdot,\cdot\rangle_K\). It therefore restricts to an isometry of
	\(N_K^{n+1}\), fixes \(P_{v,s}\) pointwise and exchanges
	\(\mathcal H_{v,s}^{+}\) and \(\mathcal H_{v,s}^{-}\).
	Therefore, \(\mathcal R_{v,s}\) gives the reflection of \(N_K^{n+1}\)
	across the totally geodesic hyperplane \(P_{v,s}\). This explicit representation allowed by the linear model will be convenient in our following analysis.
	
	For fixed $v$ and variable $s$, the hypersurfaces \(P_{v,s}\) are the standard families used in the
	Alexandrov moving-plane method in space forms \cite{Alexandrov}. A detailed account of the construction in the hyperbolic case, with its application to geometric flows, has been recently given in 
	\cite[Sections~2.1 and~3.1]{BourniEspinarMishra}; for a unified treatment
	of moving planes in hyperbolic space and in a spherical hemisphere, see
	\cite[Section~2]{CiraoloRoncoroniVezzoni}.

	For any closed embedded hypersurface \(M\subset N_K^{n+1}\ \), we can give the definitions regarding the reflection of $M$ exactly as in Definition \ref{def:reflection}, by replacing the Euclidean hyperplanes with the totally geodesic hypersurfaces $P_{v,s}$. Then Theorem \ref{thm:CG} by Chow-Gulliver on the preservation of 
	strict reflection along the flow holds unchanged in the space form setting.
	
	We now study the relation between the $C^1$-closeness to a sphere of our initial data and the reflection property. We follow here a different strategy from the one of the Euclidean case: instead of analyzing $M_0$ directly as in Proposition \ref{reflect}, we first derive estimates for the reflection of a round sphere and then extend the result to $M_0$ by a perturbative argument.

	\begin{lemma}[Reflection of the reference sphere]
		\label{lem:reference-sphere-reflection}
		Let
		\(
		S_{R_0}^K(o)=\partial B_{R_0}^K(o).
		\)
		In the spherical case, assume that \(R_0<\pi/2\). Then, for every \(v\in \mathbb S_o^n\) and \(s\in(0,R_0]\), the sphere $S_{R_0}^K(o)$ reflects strictly across $P_{v,s}$. Along the reflection hypersurface we have the transversality relation
		\begin{equation}\label{eq:reference-sphere-angle}
		\left\langle\partial_r,N_{v,s}\right\rangle_K
		=
		\frac{s_K(s)}{s_K(R_0)}>0 \quad \text{on} \quad 	S_{R_0}^K(o)\cap P_{v,s}.
		\end{equation}
		
		In addition, the reflection is uniform for $P_{v,s}$ away from the origin and for $x$ away from $P_{v,s}$, in the following sense. For any \(\gamma>0\) and \(\rho>0\), there exists 
		\(\delta_{\gamma,\rho}>0\) such that, for every \( s \in [\gamma,R_0] \) and $v \in \mathbb{S}^n_o$, we have
		\begin{equation}\label{eq:uniform-radial-separation}
		r\bigl(\mathcal R_{v,s}(x)\bigr)
		\le R_0-\delta_{\gamma,\rho}
		\end{equation}
		for every
		\(x\in S_{R_0}^K(o)\cap\mathcal H_{v,s}^{+}\) such that
		\(
		\dist(x,P_{v,s})\ge\rho.\) 
	\end{lemma}
	
	\begin{proof}
		If $s=R_0$, then $S_{R_0}^K(o)$ and $P_{v,s}$ touch at a single point, and the assertions of the theorem are either vacuous or trivial. We then consider $s \in (0,R_0)$ and take
		\(
		x\in S_{R_0}^K(o)\cap\mathcal H_{v,s}^{+}.
		\)
		We set
		\[
		a=\langle x,N_{v,s}\rangle_K,
		\qquad
		x^*=\mathcal R_{v,s}(x),
		\]
		where $a>0$ since $x\in \mathcal H_{v,s}^{+}$. Observe that, in our model, the distance from \(o\) is characterized by
		\begin{equation}\label{eq:space-form-distance-linear-model}
		c_K(r(x))=K\langle x,o\rangle_K.
		\end{equation}
		Since \(r(x)=R_0\), equations
		\eqref{eq:linear-space-form-reflection} and
		\eqref{eq:space-form-distance-linear-model} give
		\begin{equation}
		\label{eq:reflected-radius-comparison}
		c_K(r(x^*))
		=K\langle x-2aN_{v,s},o\rangle_K =c_K(R_0)-2aK\langle N_{v,s},o\rangle_K =c_K(R_0)+2aKs_K(s),	\end{equation}
		which gives $c_K(r(x^*))-c_K(R_0)=2aKs_K(s)$.
		Taking into account the different monotonicity of $c_K(r)$ for $K=1$ and $K=-1$, 
		we conclude in both cases that $r(x^*)<R_0$, that is, $x^*$ is contained in the open ball $B_{R_0}^K(o)$.
		
		To prove \eqref{eq:reference-sphere-angle}, let us consider any point $x\in S_{R_0}^K(o)\cap P_{v,s}$.
		From \eqref{eq:space-form-geodesic} and \eqref{eq:space-form-geodesic-tangent}
		we see that there is  \(\vartheta\in \mathbb S_0^n\),
		such that
		\[
		x=c_K(R_0)o+s_K(R_0)\vartheta
		\]
		\[
		\partial_r=-Ks_K(R_0)o+c_K(R_0)\vartheta.
		\]
		Using \(c_K^2+Ks_K^2=1\), these two identities yield
		\begin{equation}\label{eq:origin-radial-decomposition}
		o=c_K(R_0)x-s_K(R_0)\partial_r.
		\end{equation}
		Taking the scalar product with \(N_{v,s}\), and using
		$\langle x,N_{v,s}\rangle_K=0$ and \eqref{product-N}
		we obtain
		\[
		-s_K(s)=-s_K(R_0)\langle\partial_r,N_{v,s}\rangle_K,
		\]
		which proves \eqref{eq:reference-sphere-angle}. That identity implies in particular that \(N_{v,s}\), which is the normal direction to the reflecting hypersurface, is transverse to $S_{R_0}^K(o)$ at the intersection points; this completes the verification of the strict reflection.
		
		We now prove the uniform radial separation. By rotational symmetry, it suffices to consider a fixed $v \in \mathbb{S}^n_o$. For any given $\gamma,\rho>0$, consider the set
		\[
		\begin{aligned}
		\mathcal K_{\gamma,\rho}
		=
		\bigl\{(s,x):\;&
		\gamma\le s\le R_0,\quad
		x\in S_{R_0}^K(o) \cap \mathcal H_{v,s}^{+},\quad
		\dist(x,P_{v,s})\ge\rho
		\bigr\}.
		\end{aligned}
		\]
		If \(\mathcal K_{\gamma,\rho}\) is empty, the assertion is vacuous, and
		any positive value of \(\delta_{\gamma,\rho}\) may be chosen. Otherwise, it is easy to see that \(\mathcal K_{\gamma,\rho}\) is compact. We know from the first part of the proof that  $r(\mathcal R_{v,s}(x))<R_0$ for all $x \in \mathcal K_{\gamma,\rho}$.  Then, setting
		\(
		\delta_{\gamma,\rho}
		:=
		\min_{\mathcal K_{\gamma,\rho}}R_0-r(\mathcal R_{v,s}(x))>0
		\), 
		we obtain the sought inequality. 
	\end{proof}

	\begin{prop}[Initial reflection in space forms]
		\label{prop:initial-space-form-reflection}
		For any $c_0 \in (0,R_0)$ there exist \(\varepsilon_0, \theta_0>0\), only depending on \(R_0,c_0\)
		such that if \(M_0\) satisfies \eqref{eq:space-form-initial-annulus}--\eqref{eq:space-form-initial-angle}
		then	\(M_0\) reflects strictly  up to \(P_{v,c_0}\)
		for every \(v\in\mathbb S_o^n\). 
	\end{prop}
	
	\begin{proof}
		Let us fix \(c_0 \in(0,R_0)\) and take any $s \geq c_0$ and \(v\in\mathbb S_o^n\). As before, by rotational symmetry, it is sufficient to take a fixed $v$.  In addition, we only need to consider $s \in [s_0,R_0(1+\varepsilon_0)]$ since otherwise the outer cap of \(M_0\) is empty. We will analyze separately points of $M_0$ which are close or far from the reflecting hypersurface.
		
		In order to study the behaviour near \(P_{v,s}\), we consider signed Fermi coordinates around  \(P_{v,s}\). We denote by $z$ the coordinate corresponding to the signed distance function from \(P_{v,s}\), chosen to be positive on \(\mathcal H_{v,s}^{+}\), and by $\partial_z$ the corresponding unit tangent vector. Observe that the reflection \(\mathcal R_{v,s}\) corresponds to reversing the sign $z \mapsto -z$ and leaving the other coordinates unchanged. Observe also that, at points \( x \in P_{v,s}\), we have $\partial_z(x)=N_{v,s}$. We point out that the vector field $\partial_z$ depends on $v,s$ (in contrast with the Euclidean case, where it only depends on $v$), although we omit writing this dependence explicitly.
		
		We want to show that $\langle \nu , \partial_z \rangle >0$ at points of $M_0$ close to $P_{v,s}$. To this purpose, we first consider points on the reference sphere. By \eqref{eq:reference-sphere-angle}, we have for all $s \in [c_0,R_0]$
		$$
		\left\langle\partial_r, \partial_z  \right\rangle_K = \left\langle\partial_r, N_{v,z}  \right\rangle_K = \frac{s_K(s)}{s_K(R_0)}
		\ge
		\frac{s_K(c_0)}{s_K(R_0)}  
		>0
		\quad \text{on} \quad 	S_{R_0}^K(o)\cap P_{v,s}.
		$$

		By uniform continuity, there exists $\rho>0$ enough small, only depending on $c_0,R_0$, such that, for all $s \in [c_0,R_0]$ and for all $y \in S_{R_0}^K(o)$ with $\dist(y, P_{v,s}) \leq 3 \rho$, we have
		\begin{equation} \label{est1}
		\langle\partial_r(y), \partial_z(y) \rangle_K
		\ge
		\frac{3 s_K(c_0)}{4 s_K(R_0)}. 
		\end{equation}
	
		By the smooth dependence of $\partial_z$ on $s$, after possibly decreasing $\rho>0$ and then $\varepsilon_0>0$, we also have 
		\begin{equation} \label{est2}
		\left\langle\partial_r(x), \partial_z (x) \right\rangle_K
		\ge
		\frac{ s_K(c_0)}{2 s_K(R_0)}
		\end{equation}
		for all $s \in [c_0,R_0(1+\varepsilon_0)]$ and every $x \in N^{n+1}_K$  for which there exists $y \in S_{R_0}^K(o)$ satisfying $\dist(x,y)<\rho$ and $\dist(y, P_{v,s}) \leq 3 \rho$.
 Let us additionally require that $\varepsilon_0>0$ is small enough to have $R_0\varepsilon_0 \leq \rho$.

		With the above choice of parameters, let us now analyze the reflection of $M_0$ across $P_{v,s}$, for $s \in [c_0,R_0(1+\varepsilon_0)]$. We first consider any point $x \in M_0$ such that $\dist(x,P_{v,s}) \leq 2\rho$.  We can write $x=\exp_o\bigl(\rho_0(\vartheta)\vartheta\bigr)$ for some $\theta \in\mathbb S_o^n$. 	Let
		\(
		y=\exp_o(R_0\vartheta)
		\)
		be the point in $S_{R_0}^K(o)$ on the same radial geodesic. By assumption \eqref{eq:space-form-initial-annulus} we have
		$$
		\dist(x,y)=|\rho_0(\vartheta)-R_0| \leq R_0\varepsilon_0  \leq \rho.
		$$
		Then we also have
		\[
		\dist(y, P_{v,s})	\le
		\dist(x, P_{v,s}) + \dist(x,y) \leq 3 \rho.
		\]
		By \eqref{est2}, we conclude that
		$$
		\left\langle\partial_r(x), \partial_z (x) \right\rangle_K
		\ge
		\frac{ s_K(c_0)}{2 s_K(R_0)}>0
		$$	
		for all $x \in M_0$ such that $\dist(x,P_{v,s}) \leq 2 \rho $.

		It then follows from \eqref{eq:space-form-initial-angle} that for any $x \in M_0$ such that $\dist(x,P_{v,s}) \leq 2\rho$, we have
		\begin{equation}\label{eq:space-form-perturbed-transversality}
		\left\langle\nu_0 (x),\partial_z(x) \right\rangle_K
		\ge
		\left\langle\partial_r (x),\partial_z(x) \right\rangle_K
		-
		|\nu_0(x)-\partial_r(x)|\ge
		\frac{s_K(c)}{2s_K(R_0)}
		-
		\sqrt{2-2 \cos\theta_0} >0, 
		\end{equation}
		provided \(\theta_0\) is sufficiently small, depending only on $c_0$, $R_0$.
		
		In particular, for points $x \in  M_0 \cap P_{v,s}$,  \eqref{eq:space-form-perturbed-transversality} shows that
		$\langle \nu_0(x),N_{v,s} \rangle >0$, i.e. $N_{v,s}$ is not tangent to $M_0$ along the intersection with $P_{v,s}$.
		
		In addition, similarly to the Euclidean case, the positivity of $\langle \nu_0,N_{v,s} \rangle$ allows to exclude that a point  \(x \in M_0 \cap \mathcal H_{v,s}^{+}\) with $\dist(x,P_{v.s}) \leq 2\rho$ satisfies $\mathcal R_{v,s} (x) \in M_0$. In fact, $x$ and \(\mathcal R_{v,s} (x) \) lie on the same geodesic perpendicular to $P_{v,s}$  parametrized by $z \in [-2\rho,2\rho]$. Along such a geodesic, the scalar product $\left\langle\nu_0,\partial_z(x) \right\rangle_K$ must change sign on two consecutive transverse intersections with $M_0$, depending on whether the geodesic enters or exits the region enclosed by $M_0$ as $z$ increases. On the other hand, \eqref{eq:space-form-perturbed-transversality} shows that the sign is always positive. This implies that there can be at most one intersection, and therefore \(\mathcal R_{v,s} (M_0) \) cannot intersect $M_0$ at points at distance less than $2 \rho$ from $P_{v,s}$.
		
		We observe that, for $R_0\leq s\leq R_0(1+\varepsilon_0)$, every point
			$x\in M_0\cap \mathcal H^+_{v,s}$ satisfies 
			\[
			\operatorname{dist}(x,P_{v,s})
			\leq r(x)-s
			\leq R_0\varepsilon_0 \leq \rho  < 2 \rho.
			\]
			Thus, for such values of $s$,  the whole outer cap is covered by the preceding near-hyperplane argument.

			It remains to consider $s\in[c_0,R_0]$ and points
			$x\in M_0\cap \mathcal H^+_{v,s}$ 
			 such that $\dist(x,P_{v,s})>2\rho$. As before, we can use \eqref{eq:space-form-initial-annulus} to show that there is $y \in S_{R_0}^K(o)$ on the same radial geodesic, which satisfies $\dist(x,y) \leq R_0\varepsilon_0 \leq \rho$. Then $y$ satisfies 
		$$
		\dist(y,P_{v,s}) \geq \dist(x,P_{v,s}) - \dist(x,y) \geq \rho
		$$
		and therefore 
		$$r(\mathcal R_{v,s} (y)) \leq R_0-\delta,$$
		where $\delta=\delta(\rho,c_0,R_0)>0$ is given from Lemma \ref{lem:reference-sphere-reflection}. Since $\mathcal R_{v,s}$ is an isometry, we deduce
		$$r(\mathcal R_{v,s} (x))
		\leq  r(\mathcal R_{v,s} (y)) + \dist(x,y)
		\leq R_0-\delta+R_0\varepsilon_0.
		$$
		For $\varepsilon_0>0$ small enough to have $2 R_0\varepsilon_0 <\delta$, we conclude that $r(\mathcal R_{v,s} (x))<R_0(1-\varepsilon_0)$, so that by \eqref{eq:space-form-initial-annulus} $\mathcal R_{v,s} (x)$ lies in the interior of the region bounded by $M_0$.

		Thus, we have shown that \(M_0\) reflects strictly across every \(P_{v,s}\) with \(v\in S_o^n\) and \(s\ge c_0\). 
	\end{proof}

	We next derive the quantitative geometric consequence of the reflection
	property that will replace the Euclidean estimate for the tangential
	part of the position vector.
	Let
	\begin{equation} \label{defY}
	Y=s_K(r)\partial_r
	\end{equation}
	be the radial conformal vector field. Reflection up to distance \(c\) in
	every direction controls the tangential component of \(Y\), and hence the
	gradient of the radial graph function.
	
	\begin{lemma}[Gradient estimate from Alexandrov reflection]
		\label{lem:space-form-gradient-reflection}
		Let \(M\subset N_K^{n+1}\) be a smooth, closed and embedded hypersurface
		which reflects strictly up to \(P_{v,c}\) for every
		\(v\in\mathbb S_o^n\). 
		Then, wherever
		\(
		s_K(r)>s_K(c),
		\)
		one has
		\begin{equation}\label{eq:normal-from-reflection}
		\left\langle\nu,\partial_r\right\rangle
		\ge
		\sqrt{
			1-
			\frac{s_K(c)^2}{s_K(r)^2}
		}.
		\end{equation}
		If \(M\) is written as a radial graph
		\(
		M
		=
		\left\{
		\exp_o\bigl(\rho(\vartheta)\vartheta\bigr):
		\vartheta\in\mathbb S_o^n
		\right\},
		\)
		then
		\begin{equation}\label{eq:radial-gradient-from-reflection}
		|\nabla^\sigma\rho|_\sigma
		\le
		\frac{
			s_K(c)s_K(\rho)
		}{
			\sqrt{s_K(\rho)^2-s_K(c)^2}
		}.
		\end{equation}
	\end{lemma}
	
	\begin{proof}
		Fix \(x\in M\). If \(Y^T(x)=0\), there is nothing to prove. Assume
		\(Y^T(x)\neq0\), and set
		\(
		\tau=\frac{Y^T(x)}{|Y^T(x)|}.
		\)
		Then
		\(
		\langle\partial_r,\tau\rangle_K>0.
		\)
		
		We now use the linear model of \(N_K^{n+1}\) and consider the hyperplane $\tau^\perp \subset \mathbb E_K^{n+2}$ orthogonal to $\tau$. As 
		$\tau
		\in T_xM\subset T_xN_K^{n+1}$,
		\(x\) belongs to
		the hypersurface \(P_\tau=N_K^{n+1}\cap\tau^\perp\), which is totally
		geodesic.  Moreover, \(\tau\in T_xM\), showing that $M$ does not reflect strictly across \(P_\tau\) because the transversality condition is violated.

		Let \(s\ge0\) denote the distance from \(o\) to \(P_\tau\). Since $\tau$ is the unit normal to a totally geodesic
		hypersurface at distance \(s\) from \(o\), by \eqref{product-N} it holds
		\[
		s_K(s) = -\langle o,\tau\rangle_K = s_K(r)\langle\partial_r,\tau\rangle_K
		,\]
		where we have written $
		o=c_K(r)x-s_K(r)\partial_r
		$ as in the final part of the proof of Lemma \ref{lem:reference-sphere-reflection}.
		
		Since the reflection property does not hold, the distance of the hypersurface  \(P_\tau\) must satisfy $s<c$ otherwise we contradict our hypotheses. We conclude
		\[
		\langle\nu,\partial_r\rangle^2 = 1 - \langle\tau,\partial_r\rangle^2
		= 1 -\frac{s_K(s)^2}{s_K(r)^2}
		\ge
		1-\frac{s_K(c)^2}{s_K(r)^2}.
		\]

		The orientation is chosen so that
		\(\langle\nu,\partial_r\rangle>0\), and therefore
		\eqref{eq:normal-from-reflection} follows. Finally, for a radial graph,
		combining \eqref{eq:space-form-angle-gradient} with
		\eqref{eq:normal-from-reflection} and solving for
		\(|\nabla^\sigma\rho|_\sigma\) gives
		\eqref{eq:radial-gradient-from-reflection}.
	\end{proof}
	
	We now combine the preceding gradient estimate with volume preservation
	to control the radial position of the evolving hypersurfaces.
	\begin{prop}[Radial closeness for later times]
		\label{prop:space-form-radial-closeness}
		Let \(c_0>0\) be suitably small (depending on $R_0$) and let the initial data satisfy
		\eqref{eq:space-form-initial-annulus} and
		\eqref{eq:space-form-initial-angle}
		with $\varepsilon_0,\theta_0>0$ the values given by
		Proposition~\ref{prop:initial-space-form-reflection}. Then, for every time for which the
		solution is defined, we have
		\begin{equation}\label{eq:space-form-radial-closeness}
		\big|r(x) - R_0\big|
		\le 2\pi s_K(c_0)
		\qquad\text{for every }x\in M_t.
		\end{equation}
		In particular, in the case $K=1$, if $\varepsilon_0,\theta_0$ are suitably small then the solution remains inside the half-sphere of radius $\pi/2$ centred at $o$.
	\end{prop}
	
	\begin{proof}
		The strict reflection property is preserved by the flow by Theorem \ref{thm:CG}. Then we can apply the gradient estimate from
		Lemma~\ref{lem:space-form-gradient-reflection} for every time in the existence interval.
		
		In the case $K=1$, let us set
		\[
		T^*
		=
		\sup
		\left\{
		T>0:
		\frac{R_0}{2} < r < \frac \pi2 
		\text{ on } M_t, \text{ for every } t\in[0,T)
		\right\}.
		\]
		In the case $K=-1$, we define $T^*$ analogously, but requiring only the bound from below $\frac{R_0}{2} < r$.
		The initial annular assumption, together with the condition $R_0(1+\ep_0)<\pi/2$ in the $K=1$ case, implies by continuity  that \(T^*>0\).

		By requiring \(c_0\) to be
		sufficiently small, we may assume that
		\(
		s_K(c_0)
		\leq
		\sqrt{\frac34}s_K(R_0/2).
		\)
		Then on \([0,T^*)\) estimate
		\eqref{eq:radial-gradient-from-reflection} gives
		\[
		|\nabla^\sigma\rho|_\sigma
		\le
		\frac{
			s_K(c_0)s_K(\rho)
		}{
			\sqrt{
				s_K(\rho)^2-s_K(c_0)^2}} \le
		2 s_K(c_0).
		\]
		Since the diameter of \((\mathbb S_o^n,\sigma)\) is \(\pi\),
		we conclude
		\begin{equation}\label{eq:radial-oscillation-direct}
		\rho_+(t)-\rho_-(t)
		\le
		2 \pi \, s_K(c_0),
		\end{equation}
		where
		\(
		\rho_-(t)=\min_{\mathbb S_o^n}\rho(\cdot,t),\)
		\(
		\rho_+(t)=\max_{\mathbb S_o^n}\rho(\cdot,t).
		\)
		
		We now use volume preservation to locate this annulus. Since
		\(
		\operatorname{vol}(\Omega_t)
		=
		\operatorname{vol}\bigl(B_{R_0}^K(o)\bigr)
		\)
		and \(M_t\) is a radial graph, the monotonicity of the volume of
		geodesic balls implies
		\[
		\rho_-(t)\le R_0\le\rho_+(t).
		\]
		Combining this with \eqref{eq:radial-oscillation-direct}, we obtain
		\begin{equation}\label{eq:radial-closeness-bootstrap}
		R_0-2 \pi \, s_K(c_0)
		\le
		\rho(\vartheta,t)
		\le
		R_0+2 \pi \, s_K(c_0).
		\end{equation}
		We choose $c_0$ so small that  
		\(
		R_0-2 \pi s_K(c_0)>\frac{R_0}{2}
		\). In the $K=1$ case, we also require $R_0+ 2 \pi \, s_K(c_0)< \pi/2$.
		Suppose \(T^*<T_{\max}\). Then the
		smoothness of the flow implies that
		\eqref{eq:radial-closeness-bootstrap} also holds at $t=T^*$.
		Then we have
		\[
		\frac{R_0}{2} < \min_{M_{T^*}}r \leq \max_{M_{T^*}}r <  \frac{\pi}{2},
		\]
		where the last inequality only concerns the $K=1$ case. 
		By continuity, the strict inequalities also hold up to a larger time, contradicting the maximality of $T^*$.
		This shows that \(T^*=T_{\max}\) and that our radial bound holds for every time for which the solution exists.
	\end{proof}
	
	\begin{cor}[Normal control for later times]
		\label{cor:space-form-normal-control}
		Let our solution satisfy the reflection property with threshold $c_0$. Then, for every time for which it is defined, we have $\langle\nu,\partial_r\rangle\geq \sigma_*$, for some $\sigma_*=\sigma_*(c_0,R_0)>0$. Moreover, for any $R_0>0$, we have $\sigma_*(c_0,R_0) \to 1$
		as $c_0\to 0$.
	\end{cor}
	
	\begin{proof}
		The result follows directly from \eqref{eq:normal-from-reflection}
		and the lower bound on $r$ given by 
		Proposition~\ref{prop:space-form-radial-closeness}.
	\end{proof}

	\begin{cor}[Reflection and geometric \(C^1\)-closeness in space
		forms]
		\label{cor:space-form-reflection-C1-equivalence}
		Let \(K\in\{-1,1\}\). Within the volume-normalized and hemispherically
		localized class considered above, strong Alexandrov reflection in every
		direction and geometric \(C^1\)-closeness to a geodesic sphere
		determine each other quantitatively, up to a change of parameters.
		
		More precisely, sufficiently small annular width and normal oscillation
		give the common reflection threshold of
		Proposition~\ref{prop:initial-space-form-reflection}. Conversely, a
		common threshold tending to zero implies the radial and normal estimates
		of Proposition~\ref{prop:space-form-radial-closeness} and
		Corollary~\ref{cor:space-form-normal-control}, whose deviations from the
		reference geodesic sphere also tend to zero.
	\end{cor}
	
	\begin{proof}
		The first implication is
		Proposition~\ref{prop:initial-space-form-reflection}. For the converse,
		Lemma~\ref{lem:space-form-gradient-reflection} gives the tangential and
		normal control from the reflection threshold, while the
		volume-normalized radial estimate follows from the argument of
		Proposition~\ref{prop:space-form-radial-closeness}.
	\end{proof}
	
	The preceding corollary is the space-form counterpart of Corollary~\ref{cor:reflection-C1-equivalence}. The difference is that the initial reflection threshold is obtained perturbatively from the strict reflection property of the reference geodesic sphere, rather than from the explicit Euclidean formula \eqref{size_c0}. Once established, however, reflection provides the same type of first-order geometric control required in the curvature estimates.
	
	\subsection{Curvature estimates}
	\label{subsec:space-form-curvature-estimates}
	
	We now extend the curvature estimates of Section~\ref{sec:curvature-control}
	to the non-flat space forms. The proof is the same as in the Euclidean
	case, except for the presence of ambient-curvature terms. Since
	\(K\in\{-1,1\}\) is constant and the hypersurfaces remain in a fixed
	compact annulus with a uniform lower bound for
	\(\langle\nu,\partial_r\rangle\), all these additional terms are of
	lower order and can be absorbed into the constants. We only indicate the
	necessary modifications.
	
	With the notation from \eqref{defY}, the support function reads as
	\(
	u=\langle Y,\nu\rangle.
	\)
	By Proposition~\ref{prop:space-form-radial-closeness} and
	Corollary~\ref{cor:space-form-normal-control}, there exist constants
	\(r_*,r^*>0\) and \(\sigma_*>0\) such that
	\begin{equation}\label{eq:space-form-uniform-C1-control}
	r_*\le r\le r^*,
	\qquad
	\langle\nu,\partial_r\rangle\ge\sigma_*
	\end{equation}
	throughout the existence interval. In the spherical case, we also have
	\(r^*<\pi/2\). Consequently,
	\begin{equation}\label{eq:space-form-support-bounds}
	0<u_*\le u\le u^*<\infty.
	\end{equation}
	
	The evolution equations of the main geometric quantities take the following form, see, for example,
	\cite{Hu-sph, CabezasRivasMiquel, McCoyMixedVolume}.
	\begin{lemma}[Evolution equations in space forms] \label{evol-SF} Under \eqref{vp-mcf} it holds
		\begin{enumerate}
			\item[{\rm (a)}] $(\partial_t-\Delta)H
			=
			(H-h)(|A|^2+nK)$. \smallskip
			
			\item[{\rm (b)}]
			$(\partial_t-\Delta)|A|^2
			= 
			-2|\nabla A|^2+2|A|^2\big(|A|^2 - n K\big)
			+4KH^2-2 h \big(\!\operatorname{tr}(A^3) + K H\big).$ \smallskip
			
			\item[{\rm (c)}] $(\partial_t-\Delta)u
			=
			c_K(r)(h-2H)+|A|^2u.$
		\end{enumerate}
	\end{lemma}

	\begin{prop}[Uniform mean curvature bound]
		\label{prop:space-form-mean-curvature-bound}
		There exists a constant \(C<\infty\) such that
		\[
		|H|\le C
		\]
		for every time for which the solution is defined.
	\end{prop}
	
	\begin{proof}
		Choose \(\alpha>0\) such that
		\(
		u-\alpha\ge\eta_*>0
		\)
		along the flow, and define
		\[
		Z=\frac{H}{u-\alpha}.
		\]
		Using Lemma \ref{evol-SF} (a) and (c), we obtain
		\begin{align}
		(\partial_t-\Delta)Z
		={}&
		-\frac{\alpha Z+h}{u-\alpha}|A|^2
		-\frac{h}{u-\alpha}Z
		+2Z^2 +\frac{2}{u-\alpha}\langle\nabla u,\nabla Z\rangle
		+E_K,
		\label{eq:space-form-Z-euclidean-part}
		\end{align}
		where 
		\[
		E_K
		=
		nKZ-\frac{nKh}{u-\alpha}
		-
		\frac{(c_K(r)-1)h}{u-\alpha}Z
		+
		2(c_K(r)-1)Z^2,
		\]
		and the remaining terms on the right-hand side of \eqref{eq:space-form-Z-euclidean-part} are the same as in the Euclidean calculation.  By
		\eqref{eq:space-form-uniform-C1-control} and
		\eqref{eq:space-form-support-bounds}, the coefficients in \(E_K\) are
		uniformly bounded, hence
		\begin{equation}\label{eq:space-form-Z-error}
		|E_K|
		\le
		C\bigl(|Z|+Z^2+|h|(1+|Z|)
		\bigr).
		\end{equation}

	 Writing $\delta:=\|u-s_K(R_0)\|_{L^\infty(M_t)},$ the space form version of the Minkowski formula gives $|h|\leq C\bigl(1+\delta\|H\|_{L^\infty(M_t)}\bigr)$, as in Lemma \ref{bound-h}. By Proposition~\ref{prop:space-form-radial-closeness} and
			Corollary~\ref{cor:space-form-normal-control}, $\delta$ can be made arbitrarily small by choosing the initial parameters sufficiently small.
			Hence the part of $h$ which grows with $\|Z\|_{L^\infty}$ has an arbitrarily small coefficient, while the remaining contribution is of lower order.

		Thus the leading term remains the one  involving \(Z |A|^2\), because of the estimate
		\[
		|A|^2\ge\frac{H^2}{n}
		=
		\frac{(u-\alpha)^2Z^2}{n}.
		\]
		This term provides the same cubic terms in the upper and lower
maximum-principle estimates as in the Euclidean argument, while the additional
		\(K\)-terms are of lower order.

		  The maximum-principle argument of Proposition~\ref{bound-H} therefore applies, with only a change in the constants. In this way we obtain uniform bounds from above and below on $Z$, which in turn yield the desired assertion on $H$.
	\end{proof}

	\begin{prop}[Uniform bound for the shape operator]
		\label{prop:space-form-shape-operator-bound}
		There exists a constant \(C<\infty\) such that
		\[
		|A|^2\le C
		\]
		for every time for which the solution is defined.
	\end{prop}
	
	\begin{proof}
		Setting
		\(
		q=\langle\nu,\partial_r\rangle\) and 
		$v=q^{-1}$, \eqref{eq:space-form-uniform-C1-control} ensures that
		\[
		1\le v\le v_*:=\sigma_*^{-1}.
		\]
		The evolution equation for \(v\) has the same principal terms as in the
		Euclidean case:
		\begin{equation}\label{eq:space-form-v-evolution}
		(\partial_t-\Delta)v
		=
		-|A|^2v-2v^{-1}|\nabla v|^2
		+\langle\mathcal X, \nabla v\rangle +\mathcal R_K,
		\end{equation}
		where \(\mathcal X\) is a uniformly bounded tangential vector field and
		\(\mathcal R_K\) contains the terms involving \(H\), \(h\), the radial
		coefficients \(s_K\), \(c_K\), and the curvature \(K\).  
		Notice that no term of order higher than \(|A|^2v\) is introduced by the
		ambient curvature.  More precisely, $\mathcal X$ and $R_K$ depend only on the uniformly controlled quantities $r$, $v$, $H$, and $h$, and on the bounded radial coefficients $s_K(r)$ and $c_K(r)$. In particular, they contain no uncontrolled component of $A$, and hence \[ |\mathcal X|+|\mathcal R_K|\le C. \]
		
		Choose \(k>0\) such that \(kv_*^2<1\), and define $\mathscr F$ as in \eqref{def-EH}. 
		As in the Euclidean case, \(\mathscr F\) is uniformly equivalent to \(|A|^2\).
		Combining the evolution from Lemma \ref{evol-SF} (b) and
		\eqref{eq:space-form-v-evolution}, computing as in the proof of
		Proposition~\ref{shape-bdd}, we reach 
		\begin{equation}\label{eq:space-form-F-evolution}
		(\partial_t-\Delta) \mathscr F
		\le
		-c \mathscr F^2+C \mathscr F+C
		+ \langle\mathcal X_1,\nabla \mathscr F\rangle
		-c_1|\nabla v|^2 \mathscr F,
		\end{equation}
		for uniform constants \(c,c_1>0\) and \(C<\infty\). Indeed, the
		additional terms in Lemma \ref{evol-SF} (b) satisfy
		\[
		\left|
		-2nK|A|^2+4KH^2-2KhH
		\right|
		\le
		C(1+|A|^2),
		\]
		and therefore contribute at most \(C(\mathscr F+1)\) to
		\eqref{eq:space-form-F-evolution}. Similarly, the terms collected in
		\(\mathcal R_K\) contribute at most \(C \mathscr F\).
		
		At a spatial maximum of \(\mathscr F\), the drift term vanishes and the term
		involving \(|\nabla v|^2 \mathscr F\) is nonpositive. Hence an equation of the form \eqref{Fmax-def} also holds in this setting. The maximum principle therefore gives a uniform bound for \(\mathscr F\),
		and hence for \(|A|^2\).
	\end{proof}
	
	Since the ambient space has constant
	curvature, the evolution equations for the covariant derivatives of
	\(A\) differ from their Euclidean counterparts only by lower-order
	contractions involving \(K\), \(A\), and its derivatives. The standard
	Bernstein estimates therefore imply that, for every integer \(m\ge0\)
	and every \(\tau>0\),
	\[
	\sup_{M_t}|\nabla^mA|
	\le
	C_{m,\tau}
	\qquad\text{for every }t\ge\tau.
	\]
	Consequently, the solution exists for every \(t\ge0\).
	
	\subsection{Subsequential convergence}
	\label{subsec:space-form-subsequential-convergence}
	
	We next use the area dissipation identity and the uniform estimates
	obtained above to identify the possible limiting hypersurfaces.
	
	\begin{prop}[Subsequential convergence]
		\label{prop:space-form-subsequential-convergence}
		There exist a sequence \(t_j\to\infty\) and a point
		\(p_*\in N_K^{n+1}\) such that
		\[
		M_{t_j}
		\longrightarrow
		S_{R_0}^K(p_*):=\partial B_{R_0}^K(p_*)
		\]
		smoothly.
	\end{prop}
	
	\begin{proof}
		The proof is the same as in the first part of 
		Proposition~\ref{prop:euclidean-eventual-convexity}. Indeed, the
		area dissipation identity is unchanged, and therefore there exists a
		sequence \(t_j\to\infty\) such that \eqref{eq:vanishing-dissipation-sequence} holds.
		
		The radial and normal estimates, together with the higher-order curvature
		bounds proved above, give smooth compactness for the radial graph
		functions
		\[
		M_t
		=
		\left\{
		\exp_o\bigl(\rho(\vartheta,t)\vartheta\bigr):
		\vartheta\in\mathbb S_o^n
		\right\}.
		\]
		Thus, after passing to a subsequence, \(M_{t_j}\) converges smoothly to
		a closed embedded constant-mean-curvature hypersurface \(M_\infty\).
		
		By Alexandrov's theorem in \(N_K^{n+1}\), \(M_\infty\) is a geodesic
		sphere. In the spherical case, the fixed-hemisphere assumption ensures
		that the corresponding version of the theorem applies. Finally, volume
		preservation and the strict monotonicity of the volume of geodesic balls
		show that its radius is \(R_0\). Hence the statement follows.
	\end{proof}

	\subsection{Eventual convexity and exponential convergence}
	\label{subsec:space-form-exponential-convergence}
	
	The preceding proposition provides a late time slice which is arbitrarily
	close in the smooth topology to a geodesic sphere. 
	
	We now complete the proof of exponential convergence and at the same time
	show that convexity is eventually attained, although no convexity or
	mean-convexity is assumed initially. Smooth subsequential convergence to a
	geodesic sphere implies that a sufficiently late time slice is uniformly
	strictly convex. In Euclidean and hyperbolic space, the conclusion then
	follows by restarting the flow and applying the corresponding convergence
	theorems for convex hypersurfaces. In the spherical case, where convexity
	is not preserved in general by \eqref{vp-mcf}, we instead use
	the evolution of \(|A|^2/H^2\). The ambient-curvature term in this
	evolution has a favourable sign near a geodesic sphere and yields
	exponential decay of the traceless second fundamental form.
	
	\begin{prop}[Eventual convexity and exponential convergence]
		\label{prop:eventual-convexity-convergence}
		Let \(K\in\{-1, 1\}\), and assume in the spherical case that
		\(R_0<\pi/2\). Then there exist \(T<\infty\) and
		\(p_\infty\in N_K^{n+1}\) such that \(M_t\) is strictly convex ($h$-convex in case $K = -1$) for every
		\(t\ge T\) and
		\[
		M_t\longrightarrow \partial B_{R_0}^K(p_\infty)
		\]
		exponentially in the smooth topology.
	\end{prop}
	
	\begin{proof}
		By Proposition~\ref{prop:space-form-subsequential-convergence}, there exist a
		sequence \(t_j\to\infty\) and \(p_*\in N_K^{n+1}\) such that
		\(
		M_{t_j}\to S_*:=\partial B_{R_0}^K(p_*)
		\)
		smoothly. Setting
		\(
		\lambda_0=\frac{c_K(R_0)}{s_K(R_0)},
		\) the Weingarten
		operator of \(S_*\) is
		\(
		h_i^j=\lambda_0\delta_i^j.
		\)
		Since \(\lambda_0>0\), smooth convergence implies that, for all
		sufficiently large \(j\),
		\[
		\kappa_i(\cdot,t_j)\ge\frac{\lambda_0}{2}>0,
		\qquad i=1,\ldots,n.
		\]
		Thus \(M_{t_j}\) is uniformly strictly convex. If \(K=-1\), then
		\(\lambda_0=\coth R_0>1\), and we may choose \(j\) so large that
		\[
		\left|
		\kappa_i(\cdot,t_j)-\coth R_0
		\right|
		<
		\frac{\coth R_0-1}{2}
		\]
		for every \(i=1,\ldots,n\). Consequently,
		\[
		\kappa_i(\cdot,t_j)
		>
		\frac{1+\coth R_0}{2}>1,
		\qquad i=1,\ldots,n,
		\]
		and \(M_{t_j}\) is strictly convex by horospheres.
		
		Fix one such time \(t_0=t_j\). If \(K=-1\), the time slice \(M_{t_0}\) is strictly convex by
		horospheres. The convergence theorem for \eqref{vp-mcf} in hyperbolic space \cite{CabezasRivasMiquel} implies that the time-shifted solution
		remains convex by horospheres and converges exponentially to a geodesic
		sphere. The enclosed volume is preserved, and therefore the limiting geodesic
		sphere has radius \(R_0\). For the exponential convergence in the hyperbolic plane, one only needs $\kappa(\cdot, t_j) > 0$ and then invoke the corresponding convexity theorem in \cite{wei_yang_VPF_Hyperbolic}.

		 It remains to prove the result when \(K=1\). In this case, we could invoke the convergence theorem by Hartley \cite{Hartley}. However, since that result relies on deep  arguments using analytic semigroups and centre manifold theory, we provide a  direct independent proof, which exploits the stronger properties of the solution obtained above.	
			
			We treat separately the cases $n \geq2$ and $n=1$.
		If $n \geq 2$, we give a direct pinching
		argument for  
		$Q=\frac{|A|^2}{H^2}$. 	
		
		Since \(M_{t_0}\) is arbitrarily close to \(S_*\) in \(C^\infty\), after
		increasing \(t_0\) we may assume that
		\begin{equation}\label{eq:initial-spherical-neighbourhood}
		\frac{\lambda_0}{2}\le\kappa_i\le2\lambda_0,\qquad
		\frac{n\lambda_0}{2}\le H,h\le2n\lambda_0,\qquad
		\frac hH\le\frac32
		\end{equation}
		on \(M_{t_0}\), where \(\lambda_0=\cot R_0\).

		By Lemma \ref{evol-SF} (a) and (b) (see also \cite[Lemma 5.1]{McCoyMixedVolume}), one has
		\begin{align}
		(\partial_t-\Delta)Q
		={}&\frac2H\langle\nabla H,\nabla Q\rangle
		-\frac2{H^4}|H\nabla A-A\otimes\nabla H|^2 \notag\\
		&+\frac{2h}{H^3}\bigl(|A|^4-H\operatorname{tr}(A^3)\bigr)
		-2n\left(2-\frac hH\right)\left(Q-\frac1n\right).
		\label{eq:space-form-Q-evolution}
		\end{align}
		Indeed, the \(K\)-terms in the quotient evolution for $K = 1$ are
		\[
		-\frac{2n |A|^2}{H^2}+4-\frac{2 h}{H}
		-\frac{2n |A|^2(H-h)}{H^3}
		=2(1-nQ)\left(2-\frac hH\right).
		\]
		
		Moreover, 	since
		\(
		Q-\frac1n=\frac{|\mathring A|^2}{H^2}
		\), if  \(\kappa_i\ge m>0\), then
		\[
		H\operatorname{tr}(A^3)-|A|^4
		=\sum_{i<j}\kappa_i\kappa_j(\kappa_i-\kappa_j)^2 \ge m^2\sum_{i<j}(\kappa_i-\kappa_j)^2
		=n\, m^2|\mathring A|^2.
		\]
		Consequently,
		\[
		\frac{2h}{H^3}\bigl(|A|^4-H\operatorname{tr}(A^3)\bigr)
		\le-\frac{2nhm^2}{H}\left(Q-\frac1n\right).
		\]

		Let \(T^*>t_0\) be the supremum of the times for which
		\[
		\frac{\lambda_0}{4}<\kappa_i<3\lambda_0,\qquad
		\frac{n\lambda_0}{4}<H,h<3n\lambda_0,\qquad
		\frac hH<\frac74
		\]
		holds on \([t_0,T^*)\). On this interval we may take
		\(m=\lambda_0/4\), and the preceding estimates yield
		\begin{equation}\label{eq:Q-coercive}
		(\partial_t-\Delta)Q
		\le\frac2H\langle\nabla H,\nabla Q\rangle
		-\mu_0\left(Q-\frac1n\right)
		\end{equation}
		for a constant \(\mu_0=\mu_0(n,R_0)>0\). At a spatial maximum of \(Q\), the
		drift term vanishes. Hence the maximum principle gives
		\begin{equation}\label{eq:Q-exponential}
		\max_{M_t} \frac{|\mathring A|^2}{H^2}
		\le e^{-\mu_0(t-t_0)}
		\max_{M_{t_0}} \frac{|\mathring A|^2}{H^2} \qquad \text{for every \quad \(t<T^*\).}
		\end{equation}

		Since \(H\) is uniformly bounded above and below on this interval,
		the latter implies
		\begin{equation}\label{eq:traceless-decay}
		\|\mathring A(\cdot,t)\|_{L^\infty}
		\le Ce^{-\gamma(t-t_0)}
		\end{equation}
		for some \(\gamma>0\). The uniform higher-order curvature estimates
		proved above, together with interpolation, give for every \(m\ge0\)
		\begin{equation}\label{eq:traceless-higher-decay}
		\|\nabla^m\mathring A(\cdot,t)\|_{L^\infty}
		\le C_m e^{-\gamma_m(t-t_0)}.
		\end{equation}
		This, by means of the Codazzi identity 
		$
		\nabla^i\mathring h_{ij}=\frac{n-1}{n}\nabla_jH,
		$
		gives $\|\nabla H(\cdot,t)\|_{L^\infty}
		\le Ce^{-\gamma(t-t_0)}$.

		By Proposition~\ref{prop:space-form-radial-closeness} and
		Corollary~\ref{cor:space-form-normal-control}, the radial graph
		functions satisfy
		\[
		0<\rho_*\le\rho(\cdot,t)\le\rho^*,
		\qquad
		|\nabla^\sigma\rho(\cdot,t)|_\sigma	=
		s_K(\rho)
		\sqrt{
			\left\langle\nu,\partial_r\right\rangle^{-2}-1
		}
		\le C.
		\]
		Since the induced metric is
		$g_t=d\rho\otimes d\rho+s_K(\rho)^2\sigma,
		$
		it follows that \(g_t\le C\sigma\), and hence
		$
		\operatorname{diam}(M_t,g_t)\le C.
		$
		Consequently,
		\[
		\operatorname{osc}_{M_t}H
		\le C\|\nabla H(\cdot,t)\|_{L^\infty}
		\le Ce^{-\gamma(t-t_0)}.
		\]
		Since \(H_{\min}\le h\le H_{\max}\), we obtain
		$\|H-h\|_{L^\infty(M_t)}
		\le Ce^{-\gamma(t-t_0)}$.
		Using \eqref{eq:traceless-higher-decay}, the evolution equations and
		parabolic interpolation once more, this improves to
		\begin{equation}\label{eq:speed-higher-decay}
		\|H-h\|_{C^m(M_t)}
		\le C_m e^{-\gamma_m(t-t_0)}
		\end{equation}
		for every \(m\ge0\). Using the radial graph parametrization fixed above, by the flow equation \eqref{vp-mcf}, the latter implies
		\[
		\int_{t_0}^{T^*}\|\partial_tX(\cdot,t)\|_{C^m}\,dt
		\le C_m\|\mathring A(\cdot,t_0)\|_{C^{m+1}}.
		\]
		Choose \(t_0=t_j\) sufficiently large that \(M_{t_0}\) belongs to a
		smaller neighbourhood whose closure is contained in the bootstrap
		region defining \(T^*\). The estimates above, with their dependence on
		the initial pinching made explicit, give
		\[
		\|\partial_tX(\cdot,t)\|_{C^m}
		\le
		C_m\varepsilon_*e^{-\gamma_m(t-t_0)},
		\]
		where
		$
		\varepsilon_*
		=
		\|\mathring A(\cdot,t_0)\|_{C^{m+1}}
		+
		\|H(\cdot,t_0)-h(t_0)\|_{C^m}
		\to 0
		$
		as \(t_0=t_j\to\infty\). Hence
		\[
		\|X(\cdot,t)-X(\cdot,t_0)\|_{C^m}
		\le
		\frac{C_m}{\gamma_m}\varepsilon_*
		\qquad\text{for every }t<T^*.
		\]
		Taking \(t_0\) sufficiently large, the solution remains at a uniform
		positive distance from the boundary of the bootstrap region. Here the bootstrap neighbourhood is understood in a fixed
		\(C^2\)-neighbourhood of \(S_{R_0}^1(p_*)\); membership in this
		neighbourhood controls \(\kappa_i\), \(H\), \(h\), and \(h/H\).
		
		In
		particular, none of the inequalities defining \(T^*\) can become an
		equality. If \(T^*<\infty\), continuity would then allow the bootstrap
		conditions to be extended beyond \(T^*\), contradicting its definition.
		Therefore \(T^*=\infty\).	In particular, \(M_t\) remains strictly convex for every \(t\ge t_0\).
		
		Equation \eqref{eq:speed-higher-decay} also shows that \(X(\cdot,t)\) is
		Cauchy in every \(C^m\)-norm. Thus it converges exponentially to a smooth
		limiting embedding \(X_\infty\). By \eqref{eq:traceless-decay}, the
		limiting hypersurface satisfies $\mathring A_\infty=0$. 
		Hence it is totally umbilical. Since it is closed, embedded and contained
		in an open hemisphere, it is a geodesic sphere
		$	M_\infty=\partial B_{R_\infty}^1(p_\infty)$.
		The enclosed volume is preserved, so
		$
		\operatorname{vol}\bigl(B_{R_\infty}^1(p_\infty)\bigr)
		=\operatorname{vol}\bigl(B_{R_0}^1(o)\bigr)
		$, and therefore \(R_\infty=R_0\). This concludes the proof when $n \geq 2$.
		
		Let us now consider the case of curves. As before, by Proposition~\ref{prop:space-form-subsequential-convergence} we have the smooth subsequential convergence
		\[
		M_{t_j}\longrightarrow
		S_*:=\partial B_{R_0}^{1}(p_*).
		\]

		We first show that the whole flow converges smoothly to $S_*$. Taking
		$p_*$ as the reference point, let $\varepsilon_j$ and $\theta_j$
		denote the annular and radial-tilt parameters of $M_{t_j}$. The
		smooth convergence above gives $
		\varepsilon_j+\theta_j\longrightarrow0$. As, by Proposition~\ref{prop:initial-space-form-reflection}, every fixed $c>0$ is a common reflection threshold for $M_{t_j}$ when $j$ is sufficiently large, we may choose common reflection thresholds $c_j\to0$.

		By Proposition~\ref{prop:space-form-radial-closeness}, the reflection threshold $c_j$ is preserved for
		every $t\geq t_j$. The quantitative converse in Corollary~\ref{cor:space-form-reflection-C1-equivalence},
		applied uniformly on the whole time interval $[t_j,\infty)$,
		therefore gives
		\[
		\sup_{t\geq t_j}
		\bigl\|\rho_{p_*}(\cdot,t)-R_0\bigr\|_{C^1}
		\longrightarrow0,
		\]
		where $\rho_{p_*}(\cdot,t)$ denotes the radial graph function of
		$M_t$ with respect to $p_*$. Since the estimate holds on the entire
		tail $[t_j,\infty)$, it yields
		\[
		\rho_{p_*}(\cdot,t)\longrightarrow R_0
		\qquad\text{in }C^1\text{ as }t\to\infty,
		\]
		without passing to a subsequence. The uniform higher-order estimates
		and interpolation then upgrade this convergence to every $C^m$
		norm. Hence
		$M_t\to S_*$ smoothly as $t\to\infty$.

		It remains to prove the exponential rate. By the definition of $h$,
		\begin{equation} \label{zero-avg}
		\int_{M_t}w\,ds=0 \qquad \text{for} \quad w:=\kappa-h.
		\end{equation}
		Moreover, if $Y$ is an ambient Killing field on $\mathbb S^2$, the
		one-parameter group of isometries generated by $Y$ preserves both
		length and enclosed area. The corresponding first-variation formulas
		give
		\begin{equation} \label{Killing-id}
		\int_{M_t}w\langle Y,\nu\rangle\,ds=0.
		\end{equation}

		On the other hand, the normal component
		of every Killing field $Y$ satisfies
		\[
		-\partial_s^2 \langle Y,\nu\rangle
		=
		\bigl(\kappa^2_{S_\ast}+\operatorname{Sec}_{\mathbb S^2}\bigr) \langle Y,\nu\rangle = (\cot^2R_0+1)  \langle Y,\nu\rangle = \frac1{\sin^2R_0}    \langle Y,\nu\rangle;
		\]
		see, for instance, \cite[Proposition~1]{FornariRipoll}. The normal components of the Killing fields obtained by moving the centre of $S_*$ in two independent directions are, up to a choice of origin and normalization, the sine and cosine functions on $S_*$. Hence,  on $M_t$, by \eqref{zero-avg} and
		\eqref{Killing-id}, $w$ is orthogonal to the constants  and to a two-dimensional space converging smoothly to this first eigenspace. Since $S_*$ is intrinsically a circle of radius $\sin R_0$, the next eigenvalue of $-\partial_s^2$ is $4\sin^{-2}R_0$; see [6, Chapter II.4]. The Rayleigh characterization (cf. \cite[Chapter~I.5]
		{ChavelEigenvalues}) and the smooth
		convergence $M_t\to S_*$ therefore imply that, for every
		$\varepsilon>0$ and all sufficiently large $t$,
		\[
		\int_{M_t}|\partial_sw|^2\,ds
		\geq
		\left(4 a_0-\varepsilon\right)
		\int_{M_t}w^2\,ds, \qquad \text{with} \quad a_0:=\frac1{\sin^2R_0}.
		\tag{6.xx}
		\]

		Now set $
		E(t):=\int_{M_t}w^2\,ds$. The evolution equations
		$
		\partial_t\kappa
		=
		\partial_s^2\kappa+(\kappa-h)(\kappa^2+1)$ and
		$
		\partial_tds=-\kappa(\kappa-h)\,ds
		$
		yield
		\[
		E'(t)
		=
		-2\int_{M_t}|\partial_sw|^2\,ds
		+
		2\int_{M_t}(\kappa^2+1)w^2\,ds
		-
		\int_{M_t}\kappa w^3\,ds,
		\]
		where the term containing $h'(t)$ vanishes because of \eqref{zero-avg}.

		Since $M_t\to S_*$ smoothly, for every $\varepsilon>0$ and all
		sufficiently large $t$, one can guarantee
		\[
		|\kappa^2+1-a_0|\leq\varepsilon,
		\qquad
		\|w\|_{L^\infty(M_t)}\leq\varepsilon,
		\qquad
		\|\kappa\|_{L^\infty(M_t)}\leq C_0,
		\]
		where $C_0$ is independent of $t$. Applying the preceding estimates, we obtain
		\[
		E'(t)
		\leq
		-2(4a_0-\varepsilon)E(t)
		+
		2(a_0+\varepsilon)E(t)
		+
		C_0\varepsilon E(t) \leq-3a_0E(t),
		\]
		after choosing $\varepsilon>0$ so that
		$
		(4+C_0)\varepsilon\leq3a_0.
		$
		Therefore, for all sufficiently large $t$,
		\[
		E(t)\leq Ce^{-3a_0t}.
		\]
		
		The uniform higher-order curvature estimates and interpolation now
		give, for every integer $m\geq0$, constants $C_m,\gamma_m>0$ such
		that
		\[
		\|\kappa(\cdot,t)-h(t)\|_{C^m(M_t)}
		\leq
		C_me^{-\gamma_mt}.
		\tag{6.xx}
		\]
		Finally, in radial coordinates with respect to $p_*$,
		$
		\partial_t\rho_{p_*}
		=
		\frac{h-\kappa}{\langle\nu,\partial_r\rangle}.
		$
		The denominator is uniformly bounded away from zero, and its
		derivatives are uniformly controlled. Since
		$\rho_{p_*}(\cdot,t)\to R_0$ smoothly, integration from $t$ to
		infinity gives
		\[
		\|\rho_{p_*}(\cdot,t)-R_0\|_{C^m}
		\leq
		C_me^{-\gamma_mt}.
		\]
		Thus $M_t$ converges exponentially in the smooth topology to $S_*$.
	\end{proof}
	
	\begin{cor}[Location of the limiting centre]
		Let \(c_0\) be as in Proposition~\ref{prop:initial-space-form-reflection}. Then
		\[
		\dist_{N_K}(o,p_\infty)\le c_0.
		\]
		In particular,
		$
		\dist_{N_K}(o,p_\infty)\to 0$
		as
		$\varepsilon_0+\theta_0\to 0$.
	\end{cor}
	
	\begin{proof}
		As in the Euclidean case, the preserved reflection property passes to
		the smooth limit. Hence the limiting ball
		$
		B_{R_0}^K(p_\infty)
		$
		reflects inside itself across every \(P_{v,s}\) with
		\(v\in\mathbb S_o^n\) and \(s\ge c_0\). This easily implies that the centre of the ball has distance at most $c_0$ from $o$.
		Since Proposition~\ref{prop:initial-space-form-reflection} implies the reflection property for every prescribed $c_0>0$, provided that $\varepsilon_0$ and $\theta_0$ are sufficiently small, the last assertion follows.
	\end{proof}

	\section{Extension to Lipschitz initial data}
	\label{sec:lip-euclidean}
	
	The annular-closeness and normal-oscillation assumptions make sense for
	radial graphs of merely Lipschitz regularity.  We prove that sufficiently
	small such data admit an instantaneously smooth volume preserving
	mean curvature flow.  The construction is geometric: smooth approximations
	inherit a common reflection threshold, and the weighted curvature estimate
	gives bounds independent of their initial curvature.

	In the following statement, we define $s_K(\cdot)$ as in \eqref{sine-cosine} when $K=\pm 1$, and $s_K(\rho)=\rho$ for $K=0$.
	
	\begin{thm}[Lipschitz stability in space forms] \label{thm:lipschitz-space-forms} Let \(K\in\{-1,0,1\}\), and let 
		\[M_0 = \left\{ \exp_o\bigl(\rho_0(\vartheta)\vartheta\bigr): \vartheta\in\mathbb S_o^n \right\} \subset N_K^{n+1}, \qquad \text{with} \qquad \rho_0\in W^{1,\infty}(\mathbb S_o^n).\] Assume that 
		\begin{equation}\label{eq:lipschitz-space-form-cone} R_0(1-\varepsilon_0) \le \rho_0 \le R_0(1+\varepsilon_0) \qquad \text{and} \qquad \frac{|\nabla^\sigma\rho_0|_\sigma} {s_K(\rho_0)} \le \tan\theta_0 \qquad\text{a.e. on }\mathbb S_o^n. 
		\end{equation} 
		When \(K=1\), require in addition that $R_0(1+\varepsilon_0)<\frac{\pi}{2}.$ If \(\varepsilon_0\) and \(\theta_0\) are sufficiently small, then
		\begin{enumerate}
			\item[{\rm (a)}] There exists a family of radial graphs $M_t = \left\{ \exp_o\bigl(\rho(\vartheta,t)\vartheta\bigr): \vartheta\in\mathbb S_o^n \right\}, \, t>0,$ which solves  \eqref{vp-mcf} classically for every \(t>0\) and is smooth for every positive time. If $\Omega_t$ denotes the region enclosed by $M_t$, then $ \operatorname{vol}(\Omega_t) = \operatorname{vol}(\Omega_0) \,\text{for every }t>0.$ Moreover,
			for all sufficiently small $t>0$,
			\[
			\label{eq:lip-initial-trace}
			\|\rho(\cdot,t)-\rho_0\|_{L^\infty(\mathbb S_o^n)}
			\leq C \sqrt{t}.
			\]
			In particular,
			$\rho(\cdot,t)\rightarrow\rho_0$ uniformly as $t\searrow0$,
			and $M_t$ converges to $M_0$ in the Hausdorff sense.
			
			\item[{\rm (b)}] 
			For every $T>0$ and every integer $m\geq0$, there exists
			$C_{m,T}<\infty$ such that
			\[
			\label{eq:lip-limit-smoothing-estimates}
			\sup_{0<t\leq T}
			t^{(m+1)/2}
			\|\nabla^m A(\cdot,t)\|_{L^\infty(M_t)}
			\leq C_{m,T}.
			\]
			
			\item[{\rm (c)}] The solution exists for all times and satisfies the same conclusions as in Theorems \ref{thm:main} and \ref{space-form-main} as $t \to \infty$. In particular, it converges to a geodesic sphere, exponentially in the $C^\infty$ topology. 
		\end{enumerate}
	\end{thm}
	
	By Rademacher's theorem, the gradient in \eqref{eq:lipschitz-space-form-cone} can be equivalently interpreted as the classical gradient at points of differentiability, or in the weak sense. By \eqref{eq:space-form-radial-normal}, we see that the requirement on $\nabla^\sigma\rho_0$ in \eqref{eq:lipschitz-space-form-cone} is equivalent to the condition  $\langle\nu_0,\partial_r\rangle \geq \cos \theta$ in \eqref{nor-osc} and \eqref{eq:space-form-initial-angle}. Observe that one could give an equivalent formulation in terms of the smallness of the Lipschitz constant of $\rho(\cdot)$ (depending on $R_0$).

	\begin{proof}
		
		The construction proceeds by smooth approximation. We give the complete argument in Euclidean space, and in the end indicate the modifications required when \(K=\pm1\).

		In order to approximate smoothly our initial datum $\rho_0\in W^{1,\infty}(\mathbb S^n)$ in a way which preserves the bounds in \eqref{eq:lipschitz-space-form-cone}, it is convenient to work first with the function \begin{equation} \label{eucl-w}
		w_0=\log\frac{\rho_0}{R_0},
		\end{equation}
		which satisfies
		\[
		\log(1-\varepsilon_0)
		\le
		w_0
		\le
		\log(1+\varepsilon_0)
		\qquad\text{and } \qquad
		|\nabla^\sigma w_0|_\sigma
		\le
		\tan\theta_0
		\qquad\text{a.e.}
		\] 
		
		Let $P_t=e^{t\Delta_\sigma}$
		denote the heat semigroup on \(\mathbb S^n\), and choose any sequence
		\(t_j\searrow0\). Set
		\[
		\widetilde w_0^{(j)}=P_{t_j}w_0.
		\]
		Then the functions $\widetilde w_0^{(j)}$ are in $C^\infty(\mathbb S^n)$, satisfy $\log(1-\varepsilon_0)
		\le
		\widetilde w_0^{(j)}
		\le
		\log(1+\varepsilon_0)$, and converge to $w_0$
		uniformly on $\mathbb S^n,$
		since \(w_0\) is Lipschitz. Moreover, as 
		$\operatorname{Ric}_\sigma=(n-1)\sigma$, well known properties of the heat kernel (cf. \cite[Theorem 1.3 and Corollary 1.4]{Sturm}) imply that
		\[
		\|\nabla^\sigma\widetilde w_0^{(j)}\|_{L^\infty}
		\le
		e^{-(n-1)t_j}
		\|\nabla^\sigma w_0\|_{L^\infty}
		\le
		\tan\theta_0.
		\]

		We then define the functions
		\(
		\widetilde\rho_0^{(j)}
		=
		R_0e^{\widetilde w_0^{(j)}},
		\)
		which are smooth approximants of $\rho_0$. We multiply them by a factor	
		\(a_j>0\), chosen so that
		\(
		\rho_0^{(j)}:
		=
		a_j\widetilde\rho_0^{(j)}
		\)
		encloses the same volume as \(M_0\). Then it is easy to check that $a_j \to 1$, and that the functions $\rho_0^{(j)}$ converge uniformly to $\rho_0$. In addition, they satisfy
		\begin{equation}\label{eq:lip-approximation-annular-bound-1}
		R_0(1-\varepsilon_0 - \delta_j)
		\le
		\rho_0^{(j)}
		\le
		R_0(1+\varepsilon_0 + \delta_j)
		\end{equation}
		for suitable numbers \(\delta_j\searrow0\) and,  for every sufficiently large \(j\),
		\begin{equation}\label{eq:lip-approximation-annular-bound-2}
		\big| \nabla^\sigma\log\rho_0^{(j)}
		\big|_\sigma
		\le
		\tan \theta_0.
		\end{equation}

		For each $j \geq 1$, we now solve the flow \eqref{vp-mcf} with the smooth initial data \(M_0^{(j)}\) and denote by \(M_t^{(j)},\) the corresponding smooth solutions of \eqref{vp-mcf}. Using properties 
		\eqref{eq:lip-approximation-annular-bound-1} and \eqref{eq:lip-approximation-annular-bound-2}, we can apply our previous results  to obtain uniform $C^1$-bounds on the solutions.
		
		In fact, let \(c_0\) be the reflection threshold fixed in Remark~\ref{rem:choice-reflection-threshold}. As inequality \eqref{size_c0} is strict, from Proposition \ref{reflect} we obtain that \(c_0\) is a common reflection threshold for  \(M_0^{(j)}\) for every sufficiently large \(j\). Then we can apply the subsequent results of Section \ref{C1control}, to show that the reflection property is preserved along the flow, and the solutions satisfy, for every \(j\) and \(t\),
		\begin{equation}\label{eq:lip-uniform-radial-control}
		R_0-2c_0
		\le
		|x^{(j)}|
		\le
		R_0+2c_0,
		\end{equation}
		\begin{equation}\label{eq:lip-uniform-normal-control}
		q^{(j)}:=\left\langle
		\nu^{(j)},\frac{x^{(j)}}{|x^{(j)}|}
		\right\rangle
		\ge
		q_*(c_0)
		:=
		\left(
		1-
		\frac{c_0^2}{(R_0-2c_0)^2}
		\right)^{1/2}
		>0.
		\end{equation}
		By the results of Sections~\ref{sec:curvature-control} and \ref{LTEconv}, we obtain that each \(M_t^{(j)}\) exists globally and converges to a sphere. We want to show that these solutions converge, as $j \to \infty$, to a limit flow \(M_t\) which satisfies the conclusions of our theorem. To this purpose, the crucial step is to obtain uniform curvature bounds on \(M_t^{(j)}\).

		To estimate the curvature, we will adapt the strategy of Section \ref{sec:curvature-control} in a way that the bounds are independent of the curvature of the initial data. This time we no longer prove a separate bound on $H$ as in Proposition \ref{bound-H}; instead, we proceed directly to estimating $|A|^2$ by a suitable adaptation of the technique of Proposition~\ref{shape-bdd}. 
		
		We first observe that, by the choice of \(c_0\) in
		Remark~\ref{rem:choice-reflection-threshold}, the constant in \eqref{eq:lip-uniform-normal-control} satisfies \(q_*(c_0)\to1\) as \(c_0/R_0\to0\). We may therefore
		fix $k$ and choose the initial geometric parameters sufficiently small that
		\begin{equation}\label{eq:lip-weight-parameter}
		\frac{\sqrt n}{1+\sqrt n}
		<
		k
		<
		q_*(c_0)^2.
		\end{equation}
		Set
		\(
		v=(q^{(j)})^{-1}
		\) 
		and define $\mathscr F$ as in \eqref{def-EH}.

		By \eqref{eq:lip-uniform-normal-control}, it holds
		$
		1\le v\le q_*(c_0)^{-1}
		$
		and hence \(kv^2<1\). Thus \(\mathscr F\) and \(|A|^2\)
		are uniformly equivalent; in particular, 
		$
		\mathscr F
		\ge
		\varphi(1)|A|^2$. 
		
		Now let
		$\mathscr G_j(t)
		:=
		\max_{M_t^{(j)}}F$, 
		and compute hereafter at a point where the maximum is attained. Arguing as in the proof of Proposition \ref{shape-bdd}, but without using any bound for $|h|$ and $|H|$, we obtain
		\begin{align} \label{equation-G}
		D^+\mathscr G_j
		\le
		-2k\mathscr G_j^2
		+C(1 + |h|+ |H|)\mathscr G_j                                    
		-
		2h \operatorname{tr}(A^3)\varphi(v^2),
		\end{align}
		in the upper
		Dini-derivative sense. 
		Here \(C\) depends only on \(n\), \(R_0\), and the uniform
		first-order bounds in \eqref{eq:lip-uniform-radial-control} and
		\eqref{eq:lip-uniform-normal-control}. In particular, \(C\) is
		independent of \(j\) and of the curvature of \(M_0^{(j)}\).

		For every \(t\in(0,T_j)\),
		\begin{equation}\label{eq:lip-h-by-Fmax}
		\begin{aligned}
		|h(t)|
		&\le
		\|H(\cdot,t)\|_{L^\infty(M_t^{(j)})}\le
		\sqrt n\,
		\|A(\cdot,t)\|_{L^\infty(M_t^{(j)})} \le
		\sqrt{n(1-k)}\,
		\mathscr G_j(t)^{1/2}.
		\end{aligned}
		\end{equation}
		Then the last term in \eqref{equation-G}, which has an indefinite sign, can be estimated as follows:
		\begin{align}
		2|h|\,|\operatorname{tr}(A^3)|\varphi(v^2)
		&\le
		2\sqrt{n(1-k)}\,\mathscr G_j^{1/2}
		\frac{\mathscr G_j^{3/2}}{\sqrt{\varphi(1)}}
		=
		2\sqrt n(1-k)\mathcal G_j^2.
		\label{eq:lip-critical-nonlocal-term}
		\end{align}

		Combining the above estimates, we reach
		\begin{align}
		D^+\mathscr G_j
		\le{}&
		-2\bigl(k-\sqrt n(1-k)\bigr)\mathscr G_j^2 +
		C\mathscr G_j (1 +
		C\mathscr G_j^{1/2}).
		\label{eq:lip-Fmax-preliminary-ODE}
		\end{align}
		By \eqref{eq:lip-weight-parameter}, we can 
		choose
		$
		0<\eta<
		k-\sqrt n(1-k)
		$ so that 
		Young's inequality gives
		$
		C\mathscr G_j^{3/2}
		\le
		2\eta\mathscr G_j^2+C_\eta.
		$
		Consequently, there exist constants \(\mu>0\) and \(C<\infty\),
		independent of \(j\), such that
		\begin{equation}\label{eq:lip-Fmax-ODE}
		D^+\mathscr G_j
		\le
		-\mu\mathscr G_j^2
		+
		C(\mathscr G_j+1).
		\end{equation}
		
		Choose \(M_*>0\), independently of \(j\), such that $C(s+1)\le\frac{\mu}{2}s^2$ for every $s\ge M_*$. Then we get $D^+\mathscr G_j \le -\frac{\mu}{2}\mathscr G_j^2$ whenever $\mathscr G_j\ge M_*.$ A comparison argument gives $\mathscr G_j(t) \le \max\left\{ M_*, \frac{2}{\mu t} \right\}$ for $0<t<T_j$. Consequently, we conclude that
		\begin{equation}\label{eq:lip-instantaneous-curvature-estimate}
		\sup_{M_t^{(j)}}|A|^2
		\le
		C\left(1+\frac1t\right),
		\qquad
		t>0
		\end{equation}
		In particular,
		\begin{equation}\label{eq:lip-instantaneous-H-h-estimate}
		\|H(\cdot,t)\|_{L^\infty(M_t^{(j)})}
		+
		|h(t)|
		\le
		C\left(1+\frac1{\sqrt t}\right),
		\qquad
		t>0.
		\end{equation}
		The evolution equations for the covariant derivatives of \(A\), together
		with the standard interior Bernstein estimates on backward time intervals of length comparable to $\min\{t,1\}$, then give
		\begin{equation}\label{eq:higher}
		\sup_{M_t^{(j)}} \|\nabla^mA(\cdot,t)\|^2 
		\le
		C\left(1+\frac1{t^{m+1}}\right),
		\qquad
		t>0.
		\end{equation}
		All constants are independent of the approximating sequence and of the
		initial curvature. Therefore the estimates give
		smooth compactness away from $t=0$ 
		and allow to pass to the limit in the smooth approximating flows.

		Fix \(0<\tau<T<\infty\). The above estimates give uniform \(C^\infty\)-bounds for the radial functions
		\(\rho^{(j)}\) on
		$
		\mathbb S^n\times[\tau,T].
		$
		Arzelà--Ascoli's theorem and a diagonal argument therefore yield a subsequence
		which converges smoothly on every compact subset of
		\(\mathbb S^n\times(0,\infty)\) to a smooth positive function
		\[
		\rho\in
		C^\infty\bigl(\mathbb S^n\times(0,\infty)\bigr).
		\]
		Passing to the limit in the radial equation shows that the associated
		hypersurfaces \(M_t\) solve \eqref{vp-mcf} classically for every
		\(t>0\).
		
		It remains to identify the initial trace. For the radial
		parametrization,
		$
		\partial_tx^{(j)}
		=
		\partial_t\rho^{(j)}\,\partial_r.
		$
		Taking the scalar product with \(\nu^{(j)}\) and using
		\eqref{vp-mcf}, we obtain
		\[
		\partial_t\rho^{(j)}
		=
		\frac{h^{(j)}-H^{(j)}}
		{\langle\nu^{(j)},\partial_r\rangle}.
		\]
		Estimates  \eqref{eq:lip-uniform-normal-control} and 
		\eqref{eq:lip-instantaneous-H-h-estimate} give
		\[
		|\partial_t\rho^{(j)}|
		\le
		C\left(1+t^{-1/2}\right) \le C t^{-1/2}, \qquad 0<t \leq 1.
		\]

		Since the right-hand side is integrable at \(t=0\), integration in time
		yields
		\begin{equation}\label{eq:lip-approximation-initial-trace}
		\left\|
		\rho^{(j)}(\cdot,t)-\rho_0^{(j)}
		\right\|_{L^\infty}
		\le
		C \sqrt t,
		\end{equation}
		where \(C\) is independent of \(j\).	We know that $\rho_0^{(j)}
		\longrightarrow
		\rho_0
		\,\text{uniformly on }\mathbb S^n,
		$
		therefore we can pass to the limit in
		\eqref{eq:lip-approximation-initial-trace} and obtain
		\[
		\|\rho(\cdot,t)-\rho_0\|_{L^\infty}
		\le
		C\sqrt t,
		\]
		which proves \eqref{eq:lip-initial-trace}. Since $M_t$ and $M_0$ are radial graphs over the same sphere, the uniform convergence of their radial functions also implies that $M_t\to M_0$ in the Hausdorff sense as $t\searrow0$.
		
		Each approximating flow is volume-preserving, and starts from initial  which, by construction,
		enclose the same volume
		as \(M_0\). Therefore
		$
		\operatorname{vol}(\Omega_t^{(j)})
		=
		\operatorname{vol}(\Omega_0).
		$
		For every fixed \(t>0\), smooth convergence of the radial functions
		implies convergence of the enclosed volumes. Hence
		$
		\operatorname{vol}(\Omega_t)
		=
		\operatorname{vol}(\Omega_0).
		$
		In addition, estimates \eqref{eq:higher} pass to the
		smooth limit on every compact positive-time interval and yield conclusion (b) of Theorem \ref{thm:lipschitz-space-forms}.
		
		The flow $M_t$ we have constructed starts with Lipschitz data and becomes instantaneously smooth. For any 
		\(t_0>0\) the hypersurface \(M_{t_0}\) satisfies the hypotheses of the convergence result Theorem~\ref{thm:main} for smooth initial data, provided the original geometric parameters are chosen sufficiently small. We conclude that the limiting flow becomes strictly convex for all
		sufficiently large times and converges exponentially in the smooth
		topology to a round sphere enclosing the same volume as \(M_0\).
		
		The construction gives existence through a subsequential limit of smooth
		approximations. As in \cite{EckerHuiskenInterior}, no uniqueness assertion is made here for the Lipschitz
		initial-value problem; in particular, the limiting evolution may a
		priori depend on the chosen approximating sequence.

		We conclude by indicating the modifications required for
		\(K\in\{-1,1\}\). The proof follows the Euclidean construction and we only record the points at
		which the ambient geometry changes the argument.
		
		For the approximation, instead of the Euclidean function in \eqref{eucl-w}, 
		we use the conformal radial variable
		\begin{equation} \label{radial-reform}
		w_0 = \Psi_K(\rho_0)-\Psi_K(R_0), \qquad \Psi_K'(r)=\frac1{s_K(r)}, 
		\end{equation} 
		for which the condition on the normal in \eqref{eq:lipschitz-space-form-cone}
		again becomes $
		|\nabla^\sigma w_0|_\sigma \le \tan\theta_0 \qquad \text{a.e.}
		$

		The heat-semigroup approximation therefore preserves the
		same conformal-gradient bound. The enclosed volume is normalized by
		adding a spatially constant number \(b_j\) to the approximating
		conformal radial function. Indeed, since \(\Psi_K'>0\), the radial
		function
		\[
		\rho_b
		=
		\Psi_K^{-1}
		\bigl(
		\Psi_K(R_0)+\widetilde w_0^{(j)}+b
		\bigr)
		\]
		is pointwise strictly increasing in \(b\), and hence so is its enclosed
		volume. There is consequently a unique \(b_j\) for which the
		corresponding hypersurface encloses the same volume as \(M_0\).
		Uniform convergence implies \(b_j\to0\), while
		\[
		\nabla^\sigma
		\bigl(
		\widetilde w_0^{(j)}+b_j
		\bigr)
		=
		\nabla^\sigma\widetilde w_0^{(j)},
		\]
		so the volume normalization does not affect the cone condition.
		
		After discarding finitely many terms, Proposition~\ref{prop:initial-space-form-reflection}
		applies to all the approximations with a common reflection threshold.
		The preservation of reflection, Proposition~\ref{prop:space-form-radial-closeness},
		and Corollary~\ref{cor:space-form-normal-control} then give constants
		$
		0<r_*\le r^*,
		\,
		q_*>0,
		$
		independent of the approximation, such that
		\begin{equation}\label{eq:lip-space-form-geometric-control}
		r_*
		\le
		\rho^{(j)}
		\le
		r^*,
		\qquad
		\left\langle
		\nu^{(j)},\partial_r
		\right\rangle
		\ge
		q_*.
		\end{equation}
		In the spherical case, the parameters are chosen so that
		$
		r^*<\frac{\pi}{2}.
		$
		Thus all functions of the radial variable occurring in the evolution
		equations, including \(s_K(\rho)\), \(c_K(\rho)\), and the required
		reciprocals, are uniformly bounded on the range of the approximating
		flows.
		
		We can next estimate $|A|^2$ by the same maximum principle of the Euclidean case,
		since the terms involving the
		ambient curvature are of lower order. 
		The estimates on $H$ and on the derivatives of $A$ follow similarly. This allows us to apply the same compactness argument as in the Euclidean case, and to obtain a limiting smooth flow with all the desired properties. 
		This completes the proof of the convergence result for Lipschitz initial data.
	\end{proof}


\begin{thebibliography}{00}
		
		\bibitem{AlbertNiclosCabezasRivas} S. Albert-Niclòs and E. Cabezas-Rivas,
		Constrained curvature flows on pinched Hadamard surfaces, {\it J. Differential Equations} {\bf 488} (2027), Paper No. 114817, 40 pp.
	
		
		\bibitem{Alexandrov} A. D.  Alexandrov, A characteristic property of spheres, {\it Ann. Mat. Pura Appl.} {\bf 58} (1962), 303--315.
		
		\bibitem{AlFr03} N. D. Alikakos and A. Freire, {The normalized mean curvature flow for a small bubble in a Riemannian manifold}, {\it J. Differential Geometry} {\bf 64} (2003) 247--303.
		
		\bibitem{ACW21}	B. Andrews, X. Chen and Y. Wei,  Volume preserving flow and Alexandrov-Fenchel type inequalities in hyperbolic space, {\it J. Eur. Math. Soc.} {\bf 23} (2021), 2467--2509.
		
		\bibitem{AndrewsWei18}	B. Andrews and Y. Wei,  Quermassintegral preserving curvature flow in hyperbolic space, {\it Geom. Funct. Anal.} {\bf 28} (2018), 1183--1208.
		
		
		\bibitem{AndrewsWei21}	B. Andrews and Y. Wei,Volume preserving flow by powers of the k-th mean curvature
		Andrews, Ben; Wei, Yong
		{\it J. Differential Geom.} {\bf 117} (2021), 193--222.
		
		
		
		
		\bibitem{AntonopoulouKaraliSigal}
		D. Antonopoulou, G. Karali and I. M. Sigal,
		Stability of spheres under volume-preserving mean curvature flow,
		{\it	Dyn. Partial Differ. Equ.} \textbf{7} (2010), no. 4, 327--344.
		
		
		
		\bibitem{BourniEspinarMishra}
		T. Bourni, J. M. Espinar and A. Mishra,
		Aleksandrov reflection for geometric flows in hyperbolic spaces,
		{\it preprint} (2026), \url{arXiv:2602.12186}.
		
		
		\bibitem{CabezasRivasMiquel}
		E. Cabezas-Rivas and V. Miquel,
		Volume preserving mean curvature flow in the hyperbolic space,
		{\it Indiana Univ. Math. J.} \textbf{56} (2007), no.~5, 2061--2086.
		
		\bibitem{ChavelEigenvalues}
		I.~Chavel,
		\emph{Eigenvalues in Riemannian Geometry},
		Pure and Applied Mathematics, vol.~115,
		Academic Press, Orlando, FL, 1984.
		
		\bibitem{ChaoLingWang2013} X.-L. Chao, X.-R. Ling and X.-L. Wang, \newblock On a planar area-preserving curvature flow, \newblock Proc. Amer. Math. Soc. \textbf{141} (2013), no.~5, 1783--1789.
		
		
		\bibitem{Chow23} B. Chow, Aleksandrov reflection for extrinsic geometric flows of Euclidean hypersurfaces, {\it Adv. Nonlinear Stud.} {\bf 23} (2023), no.~1, Paper No. 20220034, 22 pp.
		
		\bibitem{Chow98} B. Chow, Geometric aspects of Aleksandrov reflection and gradient estimates for parabolic equations, {\it Comm. Anal. Geom.} {\bf 5} (1997), no.~2, 389--409.
		
		\bibitem{ChowGull} B. Chow and R.~D. Gulliver II, Aleksandrov reflection and geometric evolution of hypersurfaces, {\it Comm. Anal. Geom.} {\bf 9} (2001), no.~2, 261--280.
		
		\bibitem{CiraoloRoncoroniVezzoni}
		G. Ciraolo, A. Roncoroni and L. Vezzoni,
		Quantitative stability for hypersurfaces with almost constant curvature
		in space forms,
		{\it Ann. Mat. Pura Appl.} \textbf{200} (2021), 2043--2083.
		
		\bibitem{ddkk} D. De Gennaro, A. Diana, A. Kubin and A. Kubin, Stability of the surface diffusion flow and volume-preserving mean curvature flow in the flat torus, {\it Math. Ann.} {\bf 390} (2024), no.~3 4429--4461.
		
		\bibitem{EckerHuiskenInterior}
		K. Ecker and G. Huisken,
		Interior estimates for hypersurfaces moving by mean curvature,
		{\it Invent. Math.} \textbf{105} (1991), no. 3, 547--570.
		
		\bibitem{EsSi98} J. Escher and G. Simonett,  {The volume preserving mean curvature flow near spheres}, {\it Proc. Am. Math. Soc.} {\bf 126} (1998), 2789--2796.
		
		\bibitem{FornariRipoll} S.~Fornari and J.~Ripoll, Killing fields, mean curvature, translation maps, {\it Illinois J. Math.} \textbf{48} (2004), no.~4, 1385--1403.
		
		\bibitem{Gage1986} M.~E. Gage, \newblock On an area-preserving evolution equation for plane curves, \newblock in \emph{Nonlinear Problems in Geometry}, Contemp. Math. \textbf{51}, Amer. Math. Soc., Providence, RI, 1986, 51--62.
		
		\bibitem{Hartley} D.~J. Hartley, Stability of geodesic spheres in $\Bbb S^{n+1}$ under constrained curvature flows, {\it J. Evol. Equ.} {\bf 17} (2017), no.~4, 1209--1225.
		
		\bibitem{Harvie} B. Harvie, On weak inverse mean curvature flow and Minkowski-type inequalities in hyperbolic space, {\it Calc. Var. Partial Differential Equations} {\bf 65} (2026), no.~7, Paper No. 212, 26 pp.
		
		\bibitem{HuangLinZhang}
		Z. Huang, L. Lin and Z. Zhang,
		Stability of the volume preserving mean curvature flow in hyperbolic space,
		{\it Peking Math. J.} \textbf{6} (2023), 541--560.
		
		\bibitem{Hu-sph} G. Huisken, Deforming hypersurfaces of the sphere by their mean curvature, {\it Math. Z.} {\bf 195} (1987), no.~2, 205--219.
		
		\bibitem{Hu87} G. Huisken, {The volume preserving mean curvature flow}, {\it J. Reine Angew. Math.} {\bf 382} (1987) 35--48.
		
		\bibitem{HuiskenYau} G. Huisken and S.-T. Yau, Definition of center of mass for isolated physical systems and unique foliations by stable spheres with constant mean curvature, {\it Invent. Math.} {\bf 124} (1996), 281--311.
		
		\bibitem{LiVPMCF}
		H. Li,
		The volume-preserving mean curvature flow in Euclidean space,
		{\it Pacific J. Math.} \textbf{243} (2009), no.~2, 331--355.
		
		\bibitem{LiXu} S. Li, H.~W. Xu and E. Zhao, Volume preserving mean curvature flow of $L^2$-almost umbilical hypersurfaces in hyperbolic space, {\it Internat. J. Math.} {\bf 35} (2024), no.~8, Paper No. 2450029, 15 pp.
		
		\bibitem{McCoyMixedVolume}
		J. A. McCoy,
		The mixed volume preserving mean curvature flow,
		{\it Math. Z.} \textbf{246} (2004), no.~1--2, 155--166.
		
		
		
		\bibitem{MiglioranzaThesis} M. Miglioranza, \emph{The Volume Preserving Mean Curvature Flow in a Compact Riemannian Manifold}, Ph.D. thesis, University College London, 2020.
		
		\bibitem{SinestrariTenan} C. Sinestrari and J. Tenan, Volume preserving mean curvature flow of
		round surfaces in asymptotically flat spaces, {\it Ann. Henri Poincar\'e}, (2026) online first.
		
		\bibitem{Sturm} M.-K. von~Renesse and K.-T. Sturm, Transport inequalities, gradient estimates, entropy, and Ricci curvature, {\it Comm. Pure Appl. Math.} {\bf 58} (2005), no.~7, 923--940.
		
		\bibitem{Tso} K. Tso, Deforming a hypersurface by its Gauss-Kronecker curvature,  {\it Comm. Pure Appl. Math.} {\bf 38} (1985), 867--882. 
		
		\bibitem{wei_yang_VPF_Hyperbolic} Y. Wei and B. Yang, Volume preserving flows for convex curves and surfaces in the hyperbolic space, {\it J. Funct. Anal.} {\bf 283} (2022), no.~11, Paper No. 109685, 28 pp.
		
		\bibitem{XuLZ} H.~W. Xu, Y. Leng and E. Zhao, Volume-preserving mean curvature flow of hypersurfaces in space forms, {\it Internat. J. Math.} {\bf 25} (2014), no.~3, 1450021, 19 pp.
		
		
	\end{thebibliography}
\end{document}